\documentclass[a4paper,12pt]				{amsart}
\usepackage[T1]								{fontenc} 							
\usepackage[utf8]						    {inputenc} 							
\usepackage									{ellipsis} 							
\usepackage									{lmodern} 							
\usepackage									{amssymb, amsmath, amsthm}			
\usepackage									{bm, bbm}							
\usepackage[fixamsmath]						{mathtools} 						
\usepackage[shortlabels]					{enumitem} 							
\usepackage[spacing,kerning]				{microtype}
\usepackage									{geometry}
\usepackage 								{hyperref}
\usepackage[style=alphabetic,backend=biber]	{biblatex}
\usepackage									{csquotes}							
\usepackage									{todonotes,color}
\usepackage{subcaption}
\usepackage[foot]{amsaddr}

\allowdisplaybreaks
\microtypecontext{spacing=nonfrench}

\theoremstyle{plain}
\newtheorem{theorem}					{Theorem}[section]
\newtheorem{lemma}			[theorem]	{Lemma}
\newtheorem{corollary}		[theorem]	{Corollary}
\newtheorem{proposition}	[theorem]	{Proposition}

\theoremstyle{definition}

\newtheorem{example}		[theorem]	{Example}

\theoremstyle{remark}
\newtheorem*{exa*}					{Example}
\newtheorem*{rem*}					{Remark}

\DeclareMathOperator	{\IR}		{\mathbb{R}}
\DeclareMathOperator	{\IE}		{\mathbb{E}} 
\DeclareMathOperator	{\IP}		{\mathbb{P}} 
\DeclareMathOperator	{\IN}		{\mathbb{N}} 

\DeclarePairedDelimiter	{\abs}		{\lvert}	{\rvert}
\DeclarePairedDelimiter	{\norm}		{\lVert}	{\rVert}
\DeclarePairedDelimiter	{\skal}		{\langle}	{\rangle}

\newcommand				{\eins}		{\text{$\mathbbm{1}$}}
\newcommand		    	{\dpartial}	{\mathfrak{d}}

\renewcommand 			{\epsilon}	{\varepsilon}
\renewcommand			{\phi}		{\varphi}
\renewcommand			{\tilde}	{\widetilde}

\numberwithin			{equation}{section}

\begin{document}
\title{Modified log-Sobolev inequalities and two-level concentration}
\thanks{This research was supported by the German Research Foundation (DFG) via CRC 1283 ``\textit{Taming uncertainty and profiting from randomness and low regularity in analysis, stochastics and their applications}''.}
\author{Holger Sambale$^1$ and Arthur Sinulis$^1$}
\address{$^1$Faculty of Mathematics, Bielefeld University, Bielefeld, Germany}
\email{\{hsambale, asinulis\}@math.uni-bielefeld.de}

\begin{abstract}
We consider a generic modified logarithmic Sobolev inequality (mLSI) of the form $\mathrm{Ent}_{\mu}(e^f) \le \tfrac{\rho}{2} \IE_\mu e^f \Gamma(f)^2$ for \emph{some} difference operator $\Gamma$, and show how it implies two-level concentration inequalities akin to the Hanson--Wright or Bernstein inequality. This can be applied to the continuous (e.\,g. the sphere or bounded perturbations of product measures) as well as discrete setting (the symmetric group, finite measures satisfying an approximate tensorization property, \ldots).

Moreover, we use modified logarithmic Sobolev inequalities on the symmetric group $S_n$ and for slices of the hypercube to prove Talagrand's convex distance inequality, and provide concentration inequalities for locally Lipschitz functions on $S_n$. Some examples of known statistics are worked out, for which we obtain the correct order of fluctuations, which is consistent with central limit theorems.
\end{abstract}

\maketitle
\section{Introduction}\label{section:Introduction}
Concentration and one-sided deviation inequalities have become an indispensable tool of probability theory and its applications. A question that arises frequently is to bound the fluctuations of a function $f = f(X_1, \ldots, X_n)$ of many random variables (or, equivalently, a function on a product space) around its mean, and often times it is possible to prove sub-Gaussian tail decay of the form
\[
\IP\big( f(X) - \IE f(X) \ge t \big) \le C\exp\Big( - \frac{t^2}{2K^2} \Big)
\]
for some $C \ge 1$, $K^2 > 0$ and all $t \ge 0$. There are various ways to establish sub-Gaussian estimates, such as the \emph{martingale method}, the \emph{entropy method} and an \emph{information-theoretic approach}, and we refer to the monograph \cite{BLM13} for further details. 

On the other hand, in some situations it is not possible to prove sub-Gaussian tails, and a suitable replacement might be Bernstein-type
\[
\IP\big( f(X) - \IE f(X) \ge t \big) \le C\exp\Big(- \frac{t^2}{2(a+bt)} \Big)
\]
or Hanson--Wright-type inequalities
\[
\IP\big( f(X) - \IE f(X) \ge t \big) \le C\exp\Big(- \min\Big( \frac{t^2}{a}, \frac{t}{b} \Big) \Big).
\]
As both inequalities show two different levels of tail decay (the Gaussian one for $t \le ab^{-1}$ and an exponential one for $t > ab^{-1}$), we use the terminology of Adamczak (see \cite{ABW17, AKPS18}) and call inequalities of these type \emph{two-level deviation inequalities}. If a similar estimate holds for $-f(X)$ as well, we refer to these as two-level \emph{concentration} inequalities.

The purpose of this note is to give a unified treatment of some of the existing literature on two-level deviation and concentration inequalities by showing that these are implied by a modified logarithmic Sobolev inequality (mLSI for short). We prove a general theorem providing two-level deviation and concentration inequalities in various frameworks. In particular, in Section \ref{section:Applications}, we get back and partially improve a number of earlier results like \cite{BCG17} and \cite{GS16}. 

We work in a general framework which was introduced in \cite{BG99}. Consider a probability space $(\Omega, \mathcal{F}, \mu)$ and let $\IE_\mu f$ denote the expectation of a random variable $f$ with respect to $\mu$. An operator $\Gamma$ on a class $\mathcal{A}$ of bounded, measurable functions is called a \emph{difference operator}, if
\begin{enumerate}
	\item for all $f \in \mathcal{A}$, $\Gamma(f)$ is a non-negative measurable function,
	\item for all $f \in \mathcal{A}$ and $a \ge 0, b \in \IR$ we have $af + b \in \mathcal{A}$ and $\Gamma(af+b) = a \Gamma(f)$.
\end{enumerate}
At first reading, one can think of $\Gamma(f) = \abs{\nabla f}$ in the setting $\Omega = \IR^n$. However, we want to stress that we do not require $\Gamma$ to satisfy a chain rule, and $\Gamma$ does not need to be an operator in the language of functional analysis.

We say that $\mu$ satisfies a $\Gamma\mathrm{-mLSI}(\rho)$ for some $\rho > 0$, if for all $f \in \mathcal{A}$ we have
\begin{equation} \label{eqn:modLSIdef}
\mathrm{Ent}_{\mu}(e^f) \le \frac{\rho}{2} \IE_\mu \Gamma(f)^2 e^f,
\end{equation}
where $\mathrm{Ent}_{\mu}(f) = \IE_\mu f \log f - \IE_\mu f \log( \IE_\mu f)$ ($f \ge 0$) is the entropy functional. This functional inequality is well-known in the theory of concentration of measure and has been used in various works, see \cite{BG99} and the references therein.
It is well-known that if $\mu$ satisfies a $\Gamma\mathrm{-mLSI}(\rho)$, we have for any function $f \in \mathcal{A}$ such that $\Gamma(f) \le 1$,
\begin{equation}\label{firstorder}
\mu( f - \IE_\mu f \ge t ) \le \exp\Big( - \frac{t^2}{2\rho} \Big),
\end{equation}
which is a classical first order concentration of measure result yielding subgaussian concentration (cf. \eqref{eqn:rechteFlanken}). It is not hard to see that the same holds for $-f$ if $\Gamma(af) = \abs{a}\Gamma(f)$ for all $a \in \mathbb{R}$. Our first goal is to establish second order analogues of \eqref{firstorder}.

\subsection{Two-level concentration inequalities}
Our first set of results are two-level deviation inequalities for probability measures satisfying a modified logarithmic Sobolev inequality. 

\begin{theorem}\label{theorem:SecondOrderConcentration}
	Assume that $\mu$ satisfies a $\Gamma\mathrm{-mLSI}(\rho)$ for some difference operator $\Gamma$ and $\rho > 0$. Let $f, g: \Omega \to \IR$ be two measurable functions such that $\Gamma(f) \le g$ and $g$ is sub-Gaussian, i.\,e. for some $c > 0$, $C \ge 1$ and $K > 0$
	\begin{equation}\label{eqn:gSubGaussian}
	\mu(g \ge c + t) \le C\exp\Big( - \frac{t^2}{2K^2} \Big).
	\end{equation}
	Then for all $t \ge 0$ it holds
	\begin{equation}\label{eqn:devine}
	\mu\big( f - \IE_\mu f \ge t \big) \le \frac{4C}{3} \exp\Big(- \frac{1}{8} \min\Big( \frac{t^2}{\rho c^2}, \frac{t}{\sqrt{\rho} K} \Big) \Big).
	\end{equation}
	If moreover $\Gamma(af) = \abs{a} \Gamma(f)$ for all $a \in \IR$, we have
	\begin{equation}\label{eqn:conine}
	\mu\big( \abs{f - \IE_\mu f} \ge t \big) \le 2C \exp\Big(- \frac{1}{12} \min\Big( \frac{t^2}{\rho c^2}, \frac{t}{\sqrt{\rho} K} \Big) \Big).
	\end{equation}
\end{theorem}

One possible way to show sub-Gaussian concentration for $g$ in presence of a $\Gamma\mathrm{-mLSI}(\rho)$ is by the Herbst argument. This leads to the following corollary.

\begin{corollary}\label{corollary:concIneqAlteVersion}
	Assume that $\mu$ satisfies a $\Gamma-\mathrm{mLSI}(\rho)$ for some difference operator $\Gamma$ and $\rho > 0$. Let $f, g$ be two measurable functions such that $\Gamma(f) \le g$ and $\Gamma(g) \le b$. Then for all $t \ge 0$ we have
	\[
	\mu\big( f - \IE_\mu f \ge t \big) \le \frac{4}{3} \exp\Big(- \frac{1}{8\rho} \min\Big( \frac{t^2}{(\IE_\mu g)^2}, \frac{t}{b} \Big) \Big).
	\]
	If, again, $\Gamma(af) = \abs{a}\Gamma(f)$ for all $a \in \IR$, then the same bound holds for $-f$.
\end{corollary}

By elementary means (cf. \eqref{eqn:constantAdjustment}), the constant $4C/3$ can be replaced by any $C' > 1$.
It is also possible to modify our proofs in order to apply \cite[Lemma 1.3]{KZ18}, which leads to an inequality of the form
\[
\mu\big( f - \IE_\mu f \ge t \big) \le \exp\Big( -c \min \Big( \frac{t^2}{\rho(\IE_\mu g)^2 + 2b^2 \rho^2}, \frac{t}{\sqrt{2} \rho b} \Big) \Big)
\]
for some absolute constant $c$ (the same one as in \cite{KZ18}). However, this is at the cost of a weaker denominator in the Gaussian term as compared to \eqref{eqn:devine}, and so we choose to present it in the form of Theorem \ref{theorem:SecondOrderConcentration}.

If the difference operator $\Gamma$ satisfies a chain rule-type condition, we obtain the following result, especially improving some of the constants above:

\begin{proposition}\label{sq-chr}
	Assume that $\mu$ satisfies a $\Gamma\mathrm{-mLSI}(\rho)$ for some $\rho > 0$ and some difference operator $\Gamma$ which satisfies $\Gamma(g^2) \le 2g \Gamma(g)$ for all positive functions $g$. Let $f \in \mathcal{A}$ be such that $\Gamma(f) \le g$ and $\Gamma(g) \le b$. For any $t \ge 0$ it holds
	\[
	\mu(f - \IE_\mu f \ge t) \le \exp\Big( -\frac{1}{4\rho}\min\Big( \frac{t^2}{2 \IE_\mu g^2}, \frac{t}{b} \Big) \Big).
	\]
	If $\Gamma$ satisfies $\Gamma(g^2) \le 2\abs{g} \Gamma(g)$ for any $g \in \mathcal{A}$, the same bound holds for $-f$.
\end{proposition}

We will see a number of examples of such difference operators all along this paper. Obviously, one example is the usual gradient, but also many difference operators involving a positive part satisfy the property in question.

In all the above results, a possible choice of $g$ is usually given by $g = \Gamma(f)$, resulting in $\IE_\mu \Gamma(f)$ in the denominator of the Gaussian term. In this case, the second condition reads as $\Gamma(\Gamma(f)) \le b$, which can be understood as a condition on an iterated (and thus second order) difference of $f$.

In fact, Theorem \ref{theorem:SecondOrderConcentration} can be understood as a Bernstein-type concentration inequality. Indeed, it is easy to see that for all $a, b > 0$ and $t \ge 0$ we have
\[
\frac{t^2}{a^2 + bt} \le \min\Big( \frac{t^2}{a^2}, \frac{t}{b} \Big) \le \frac{2t^2}{a^2 + bt}.
\]
This leads to the following corollary.

\begin{corollary}     \label{corollary:BernsteinInequality}
	In the situation of Theorem \ref{theorem:SecondOrderConcentration}, for all $t \ge 0$ we have
	\begin{equation}
	\mu\big( f - \IE_\mu f \ge t \big) \le \frac{4C}{3} \exp\Big( - \frac{t^2}{8(\rho c^2 + \sqrt{\rho}K t)} \Big).
	\end{equation}
	If $\Gamma(a f) = \abs{a} \Gamma(f)$ for all $a \in \IR$, then the same bound holds with $f$ replaced by $-f$.
\end{corollary}

Let us remark that the use of modified LSIs allows us to prove results for some classes of measures we could not address in previous work (e.\,g. \cite{GSS18b}), e.\,g. weakly dependent measures which might not have a finite number of atoms.

Next, we show similar deviation inequalities for an important class of functions, namely \emph{self-bounded functions}. In our framework, for a difference operator $\Gamma$ we say that $f \ge 0$ is a $\Gamma-(a,b)-$self-bounded function, if \[\Gamma(f)^2 \le af + b\] 
for some constants $a, b \ge 0$. For a product measure $\mu$, there are various sources that provide deviation or concentration inequalities for self-bounded functions, see e.\,g. \cite[Theorem 2.1]{BLM00}, \cite[Th{\'e}or{\`e}me 3.1]{Rio01}, \cite[Theorem 5]{BLM03}, \cite[Corollary 1]{BBLM05}, \cite[Theorem 3.9]{Cha05}, \cite[Theorem 1]{MR06} and \cite[Theorem 1]{BLM09}. 
As many of the proofs rely on the entropy method, it is not hard to adapt them to obtain Bernstein-type deviation inequalities only requiring an mLSI, which includes many more types of measures also allowing for dependencies:

\begin{proposition} \label{proposition:SelfBounding}
	Assume that $\mu$ satisfies a $\Gamma\mathrm{-mLSI}(\rho)$ and let $f \ge 0$ be a $\Gamma-(a,b)-$self-bounded function. Then for all $t \ge 0$ we have
	\[
	\mu\big( f - \IE_{\mu}f \ge t \big) \le \exp\Big( - \frac{t^2}{2\rho(2a \IE_{\mu}f + 2b + \frac{1}{3} at)} \Big).
	\]
	If, additionally, $\Gamma(\lambda f) = \abs{\lambda} \Gamma(f)$ for all $\lambda \in \IR$, then for all $t \in [0, \IE_{\mu} f]$ it holds
	\[
	\mu\big( \IE_{\mu} f - f \ge t \big) \le \exp\Big(- \frac{t^2}{2\rho(2a \IE_{\mu}f + 2b + \frac{1}{3}at)}\Big).
	\]
\end{proposition}

As we show in Proposition \ref{proposition:dPlusModLSI}, product measures always satisfy an mLSI with respect to a certain $L^2$-type difference operator, which was also used in the works mentioned above. This is a well-known fact and was first proven in \cite{Ma00}.

\subsection{The symmetric group}
One example we especially discuss in this note is the symmetric group $S_n$ equipped with the uniform measure. To this end, we need some notations. We write the group operation on $S_n$ as $\tau \sigma$ for $\tau,\sigma \in S_n$, and denote by $\tau_{ij}$ the transposition of $i$ and $j$. We define two difference operators (on $\mathcal{A} = L^\infty(\pi_n) = \IR^{S_n}$) via
\begin{align*}
\Gamma(f)(\sigma)^2 &= \frac{1}{n} \sum_{i,j = 1}^n (f(\sigma) - f(\sigma \tau_{ij}))^2, \\
\Gamma^+(f)(\sigma)^2 &= \frac{1}{n} \sum_{i,j = 1}^n (f(\sigma) - f(\sigma \tau_{ij}))_+^2. 
\end{align*}
For our results, we will need that the symmetric group satisfies modified logarithmic Sobolev inequalities with respect to the two difference operators defined above:

\begin{proposition}\label{proposition:mLSIforSymmetricGroup}
	Let $(S_n, \pi_n)$ be the symmetric group equipped with the uniform measure. Then a $\Gamma\mathrm{-mLSI}(1)$ and a $\Gamma^+\mathrm{-mLSI}(2)$ hold.
\end{proposition}

To formulate our next result, let us recall the notion of observable diameter. In the context of $S_n$ equipped with any metric $d$, we define it by
\[
\mathrm{ObsDiam}(S_n,d) \coloneqq \max_ {\sigma \in S_n} n^{-1} \sum_{i,j} d(\sigma, \sigma \tau_{ij})^2.
\]
For some metrics, this expression can be simplified. We say that a metric is \emph{right invariant}, if for any $\pi, \sigma, \tau \in S_n$ we have $d(\pi, \sigma) = d(\pi \tau, \sigma \tau)$, and \emph{left invariant} if $d(\pi, \sigma) = d(\tau \pi, \tau \sigma)$. It is bi-invariant, if it is right and left invariant. Assuming that $d$ is left (or right) invariant, we have 
\[
\mathrm{ObsDiam}(S_n, d) = n^{-1} \sum_{i,j} d(\mathrm{id}, \tau_{ij})^2.
\]

We call a function $f: S_n \to \IR$ \emph{locally Lipschitz} with respect to $d$, if for all $\sigma \in S_n$ and $i,j \in \{1,\ldots,n\}$ we have $\abs{f(\sigma) - f(\sigma \tau_{ij})} \le d(\sigma, \sigma \tau_{ij})$.

\begin{theorem}\label{theorem:LocallyLipschitzSn}
	Let $(S_n,d)$ be the symmetric group equipped with a metric $d$ and $\pi_n$ be the uniform distribution on $S_n$. Assume that $f: S_n \to \IR$ is locally Lipschitz with respect to $d$. For all $t \ge 0$ it holds
	\begin{equation}\label{eqn:SubGaussianObsDiam}
	\pi_n(\abs{f - \IE_{\pi_n} f} \ge t) \le 2\exp\Big( - \frac{t^2}{2\mathrm{ObsDiam}(S_n,d)} \Big).
	\end{equation}
	As a consequence, we have
	\[
	\mathrm{Var}_{\pi_n}(f) \le 4 \mathrm{ObsDiam}(S_n, d).
	\]
\end{theorem}

For example, Theorem \ref{theorem:LocallyLipschitzSn} can easily recover concentration inequalities for locally Lipschitz functions with respect to the normalized Hamming distance $d_H(\sigma, \pi) = n^{-1} \sum_{i = 1}^n \eins_{\sigma(i) \neq \pi(i)}$. In this case, $\mathrm{ObsDiam}(S_n, d_H) \le 4$. We work out further examples in Subsection \ref{subsection:examplesSymmetricGroup}.

Finally, we give a proof of Talagrand's famous concentration inequality for the convex distance for random permutations by similar means as used in the proofs of the upper results. To this end, recall that for any measurable space $\Omega$ and any $\omega = (\omega_1, \ldots, \omega_n) \in \Omega^n$, we may define the convex distance of $\omega$ to some measurable set $A \subset \Omega^n$ by
$$
d_T(\omega, A) \coloneqq \sup_{\alpha \in \mathbb{R}^n: \abs{\alpha}_2 = 1} d_\alpha(\omega, A),
$$
where
$$d_\alpha(\omega, A) \coloneqq \inf_{\omega' \in A} d_\alpha(\omega, \omega') \coloneqq \inf_{\omega' \in A} \sum_{i=1}^n \abs{\alpha_i} \eins_{\omega_i \neq \omega'_i}.$$

\begin{proposition}\label{Tcde}
	For any $A \subseteq S_n$ it holds
	\begin{equation}\label{eqn:TalagrandsInequality}
	\pi_n(A) \IE_{\pi_n} \exp\Big( \frac{d_T(\cdot, A)^2}{144} \Big) \le 1.
	\end{equation}
\end{proposition}

As compared to Talagrand's original formulation (see \cite[Theorem 5.1]{Tal95}), \eqref{eqn:TalagrandsInequality} has a weaker absolute constant 144 instead of 16. It is possible to improve our own constant a bit by invoking slightly more subtle estimates but we do not seem to arrive at 16. For product measures, an inequality similar to \eqref{eqn:TalagrandsInequality} was deduced in \cite{Tal95}, a form of which with a weaker constant was proven in \cite{BLM09} with the help of the entropy method. This was extended to weakly dependent random variables in \cite{Pau14}. However, it does not seem possible to adjust the method therein to the case of the symmetric group, and so we are not aware of any proof of either of the inequalities for the symmetric group using the entropy method. In \cite{Sam17} the author has proven the convex distance inequality for the symmetric group using weak transport inequalities.

It is possible to prove a weaker version of \eqref{eqn:TalagrandsInequality} with a somewhat better constant:

\begin{proposition} \label{proposition:Talagrand}
	Let $S_n$ be the symmetric group and $\pi_n$ be the uniform distribution on $S_n$. For any set $A \subseteq S_n$ with $\pi_n(A) \ge 1/2$ and all $t \ge 0$ we have
	\begin{equation}\label{eqn:ConvexDistanceSn}
	\pi_n(d_T(\cdot, A) \ge t) \le 2 \exp\Big( - \frac{t^2}{64} \Big).
	\end{equation}
\end{proposition}

In fact, \eqref{eqn:TalagrandsInequality} implies \eqref{eqn:ConvexDistanceSn} with a constant of 144 instead of 64.

\subsection{Slices of the hypercube}
Finally, let us discuss another model for which we are able to prove a convex distance inequality similar to \eqref{eqn:TalagrandsInequality}. Given two natural numbers $n,r$ such that $r \le n$, consider the corresponding slice of the hypercube $C_{n,r} \coloneqq \{\eta \in \{0,1\}^n \colon \sum_i \eta_i = r \}$, and denote by $\mu_{n,r}$ the uniform measure on $C_{n,r}$. On $C_{n,r}$, we define the difference operators
\begin{align*}
\Gamma(f)(\eta)^2 &= \frac{2}{n} \sum_{i<j} (f(\eta) - f(\tau_{ij}\eta))^2 = \frac{2}{n} \sum_{i,j=1}^n \eta_i(1-\eta_j) (f(\eta) - f(\tau_{ij}\eta))^2, \\
\Gamma^+(f)(\sigma)^2 &= \frac{2}{n} \sum_{i<j} (f(\eta) - f(\tau_{ij}\eta))_+^2 = \frac{2}{n} \sum_{i,j=1}^n \eta_i(1-\eta_j) (f(\eta) - f(\tau_{ij}\eta))_+^2. 
\end{align*}
Here, $\tau_{ij}\eta$ switches the $i$-th and the $j$-th coordinate of the configuration $\eta$. Up to the scaling of $2/n$, $\Gamma(f)^2$ is the generator of the so-called Bernoulli--Laplace model.

As in the previous section, a modified logarithmic Sobolev inequality holds:

\begin{proposition}\label{proposition:mLSIforBernoulliLaplace}
	For $(C_{n,r}, \mu_{n,r})$ as above, a $\Gamma\mathrm{-mLSI}(1)$ and a $\Gamma^+\mathrm{-mLSI}(2)$ hold.
\end{proposition}

Using this, we may establish a convex distance inequality by means of the entropy method again:

\begin{proposition}\label{TcdeBL}
	For any $A \subseteq C_{n,r}$ it holds
	\[
	\mu_{n,r}(A) \IE_{\mu_{n,r}} \exp\Big( \frac{d_T(\cdot, A)^2}{544} \Big) \le 1.
	\]
\end{proposition}

\subsection{Outline}
In Section \ref{section:Applications} we provide various applications and concentration inequalities. This includes examples of functions on the symmetric group (Section \ref{subsection:examplesSymmetricGroup}), concentration inequalities for multilinear polynomials in $[0,1]$-valued random variables (Section \ref{subsection:MultilinearPolynomials}), as well as consequences of Theorem \ref{theorem:SecondOrderConcentration} for the Euclidean sphere and measures on $\IR^n$ satisfying a logarithmic Sobolev inequality (Section \ref{subsection:Derivations}) and for probability measures (on general spaces) satisfying an mLSI with respect to some ``$L^2$ difference operator'' (see Section \ref{subsection:WeaklyDependentMeasures}). Moreover, in Section \ref{subsection:Bernstein} we recover and extend the classical Bernstein inequality for independent random variables (up to constants).

Section \ref{section:proofs} contains all the proofs, both of the results mentioned in this section as well as in Section \ref{section:Applications}.

\section{Applications}\label{section:Applications}
Let us now describe various situations which give rise to mLSIs with respect to ``natural'' difference operators, and show some consequences of the main results.

\subsection{Symmetric group}\label{subsection:examplesSymmetricGroup}
The aim of this subsection is to show how the results from Section \ref{section:Introduction} can be used to easily obtain concentration inequalities for functions on the symmetric group. In particular, we calculate many examples of statistics for which central limit theorems were proven, and show that the variance proxy of the sub-Gaussian estimate and the true variance agree (up to a constant independent of the dimension). This provides non-asymptotic concentration results, which are consistent with the limit theorems.

First, let us introduce the following natural metrics on $S_n$:
\begin{align*}
H(\pi, \sigma) &= \sum_{i = 1}^n \eins_{\pi(i) \neq \sigma(i)} \\
D(\pi, \sigma) &= \sum_{i = 1}^n \abs{\pi(i) - \sigma(i)} \\
S(\pi, \sigma) &= \Big( \sum_{i = 1}^n \abs{\pi(i) - \sigma(i)}^2 \Big)^{1/2} \\
I(\pi, \sigma) &= \min\{ k \ge 0 : \exists k \text{ adjacent transpositions from } \sigma^{-1} \text{ to } \pi^{-1} \}.
\end{align*}
Table \ref{table} collects some basic properties of $H$, $D$, $S^2$ and $I$.

{\renewcommand{\arraystretch}{1.5}
	\begin{table}[!ht]
		\centering
		\begin{tabular}{l|c|c|c|c}
			function $d$ & invariance & mean $\IE d(\mathrm{id},\cdot)$ & $\mathrm{Var}(d(\mathrm{id},\cdot))$ & limit theorem \\
			\hline
			H & bi-invariant & $n-1$ & $1$ & $n-H \Rightarrow \mathrm{Poi}(1)$ \\
			\hline
			D & right invariant & $\frac{n^2-1}{3}$ & $\frac{(n+1)(2n^2+7)}{45}$ & CLT\\
			\hline
			$S^2$ & right invariant & $\frac{n(n^2-1)}{6}$ & $\frac{n^2(n-1)(n+1)^2}{36}$ & CLT\\
			\hline
			I & right invariant & $\frac{n(n-1)}{4}$ & $\frac{n(n-1)(2n+5)}{72}$ & CLT
		\end{tabular}
		\vspace{0.5em}
		\caption{Invariance and probabilistic properties of the four functions $H$ (Hamming distance), $D$ (Spearman's footrule), $S^2$ (Spearman's rank correlation) and $I$ (Kendall's $\tau$). This table has been extracted from information in \cite[Chapter 6]{Dia88}.}
		\label{table}
\end{table}}

\begin{example}\label{Ex1}
	In this example, we calculate the observable diameters of the metrics on the symmetric group introduced above. By Theorem \ref{theorem:LocallyLipschitzSn}, this yields concentration properties for (locally) Lipschitz functions.
	\begin{enumerate}
		\item For the Hamming distance $H$ it is clear that $H(\sigma, \sigma \tau_{ij}) = 2$, which implies $\mathrm{ObsDiam}(S_n,H) = 4(n-1)$. So, Theorem \ref{theorem:LocallyLipschitzSn} recovers a concentration result from \cite{Mau79}.
		
		The resulting variance estimate is not always sharp; for example, if we consider the function $H(\cdot, \mathrm{id})$, the variance is $1$ and not of order $n$. On the other hand, the function $G = n - H(\mathrm{id}, \cdot)$ is a locally Lipschitz function with respect to $H$, which converges weakly to a Poisson random variable. As a consequence, there cannot be an $n$-independent sub-Gaussian estimate in the class of all locally Lipschitz functions.
		\item If we define for $p \in [1,\infty)$ a distance $d_p$ on $S_n$ by the induced $\ell^p$ norm
		\[
		d_p(\sigma, \pi) = \Big( \sum_{k = 1}^n \abs{\sigma(k) - \pi(k)}^p \Big)^{\frac 1 p},
		\]
		this yields $d_p(\sigma, \sigma \tau_{ij}) = 2^{1/p} \abs{\sigma(i) - \sigma(j)}$. Consequently, recalling that
		\[
		\sum_{i \ne j} (\sigma(i)-\sigma(j))^2 = \frac{n^2(n^2-1)}{6}
		\]
		for any $\sigma \in S_n$, we have
		\[
		\mathrm{ObsDiam}(S_n,d_p) = \frac{2^{2/p}}{6} n(n^2-1).
		\]
		The case $p = 1$ gives Spearman's footrule and $p = 2$ Spearman's rank correlation.
		\item Considering Kendall's $\tau$, we can readily see that for two indices $i,j$ and any $\sigma \in S_n$ it holds $I(\sigma, \sigma \tau_{ij}) \le 2\abs{\sigma(i)-\sigma(j)}$, since $\tau_{ij} \sigma^{-1}$ can be brought to $\sigma^{-1}$ by first taking $\sigma^{-1}(i)$ to its place, and then $\sigma^{-1}(j)$. So, as above, this leads to
		\[
		\mathrm{ObsDiam}(S_n, I) \le \frac{2}{3} n(n^2 -1).
		\]
		\item In a more general setting, let $\rho: S_n \to \mathrm{GL}(V)$ be a faithful, unitary representation of $S_n$ and let $\norm{\cdot}$ be a unitarily invariant norm on $\mathrm{GL}(V)$. Then $d_{\rho}(\sigma, \tau) \coloneqq \norm{\rho(\sigma) - \rho(\tau)}$ defines a bi-invariant metric on $S_n$, and in this case we have
		\[
		\mathrm{ObsDiam}(S_n, d_\rho) = n^{-1} \sum_{i,j} \norm{\mathrm{Id} - \rho(\tau_{ij})}.
		\]
	\end{enumerate}
\end{example}

\begin{example}
	Define the random variable $f(\sigma) = S^2(\sigma, \mathrm{id}) = \sum_{i = 1}^n (\sigma(i) - i)^2$. We have
	\begin{align*}
	\Gamma(f)^2(\sigma) &= n^{-1} \sum_{i,j = 1}^n (f(\sigma) - f(\sigma \tau_{ij}))^2 
	= 4n^{-1} \sum_{i,j = 1}^n (\sigma(i) - \sigma(j))^2(i-j)^2.
	\end{align*}
	If we define the matrix $A(\sigma) = (a_{ij}(\sigma))_{i,j}$ via $a_{ij}(\sigma) = (\sigma(i) - \sigma(j))(i-j)$, then the right hand side is (up to the factor $4n^{-1}$) the squared Hilbert--Schmidt norm of $A(\sigma)$. It is clear that $\abs{A(\sigma)}_{\mathrm{HS}} = \abs{A(\sigma^{-1})}_{\mathrm{HS}}$, and one can also easily see that it is invariant under right multiplication with any transposition $\tau_{kl}$. As any permutation can be written as a product of transpositions, we can evaluate it at the identity element. Consequently,
	\[
	\Gamma(f)^2(\sigma) = 4n^{-1} \sum_{i,j = 1}^n (i-j)^4 \le \frac{4}{15} n^5.
	\]
	Using \eqref{firstorder}, this leads to the concentration inequality
	\[
	\pi_n(\abs{f - \IE_{\pi_n} f} \ge t) \le 2\exp\Big( - \frac{15t^2}{8n^5} \Big).
	\]
	Actually, the term $n^5$ is natural, as the variance of $f$ is of order $n^5$ (see the table above). Incorporating the variance of $f$ into the inequality above leads to
	\[
	\pi_n(\abs{f - \IE_{\pi_n} f} \ge \mathrm{Var}(f)^{1/2} t) \le 2\exp\Big( - \frac{t^2}{19.2} \Big),
	\]
	which yields the correct tail behavior.
\end{example}

\begin{example}
	Let us consider the $1$-Lipschitz function $f(\sigma) = I(\sigma, \mathrm{id})$. For any $t \ge 0$ we have by \eqref{eqn:SubGaussianObsDiam}, $\mathrm{Var}_{\pi_n}(f) = n(n-1)(2n+5)/72$ and Example \ref{Ex1} (3)
	\[
	\pi_n(\abs{f - \IE_{\pi_n} f} \ge \mathrm{Var}_{\pi_n}(f)^{1/2} t) \le 2 \exp\Big( - \frac{t^2}{48} \Big),
	\]
	which is consistent with the central limit theorem for $f$.
\end{example}

\begin{example}
	We define the number of ascents $f(\sigma) = \sum_{j = 1}^{n-1} \eins_{\sigma(j+1) > \sigma(j)}.$ It can be easily shown that for any $i \neq j$ the number of ascents is not sensitive to transpositions in the sense that $\abs{f(\sigma) - f(\sigma \tau_{ij})} \le 2$.
	Consequently, this leads to $\Gamma(f)^2 \le 4(n-1)$, implying the concentration inequality
	\[
	\pi_n(\abs{f - \IE_{\pi_n} f} \ge t) \le 2 \exp\Big( - \frac{t^2}{8(n-1)} \Big),
	\]
	again using \eqref{firstorder}. Alternatively, this also follows from Example \ref{Ex1} (1). Again, the variance term of order $\sqrt{n}$ is of the right order, as in \cite{CKSS72} the authors have shown a central limit theorem for the number of ascents. More precisely, the sequence $g_n = (f - \IE_{\pi_n} f)/(\sqrt{(n+1)/12})$ converges to a standard normal distribution. The above calculations lead to
	\[
	\pi_n(\abs{g_n} \ge t) \le 2\exp\Big( - \frac{t^2}{96} \Big).
	\]
\end{example}

\begin{example}
	A closely related statistic is given by the sum of the ascents defined as $f(\sigma) = \sum_{j = 1}^{n-1} (\sigma_{i+1} - \sigma_i)_+$. A short calculation shows
	\[
	\Gamma(f)^2 = n^{-1} \sum_{i \neq j} (f(\sigma) - f(\sigma \tau_{ij}))^2 \le 4(n-1)^2 n^{-1} \sum_{i \neq j} = 4(n-1)^3.
	\]
	Indeed, if we let $\Delta_{i,j} \coloneqq (\sigma(i) - \sigma(j))_+$, then
	\begin{align*}
	&(f(\sigma) - f(\sigma \tau_{ij}))^2 \\
	&= (\Delta_{i,i-1} + \Delta_{i+1,i} + \Delta_{j,j-1} + \Delta_{j+1,j} - \Delta_{j,i-1} - \Delta_{i+1,j} - \Delta_{i,j-1} - \Delta_{j+1,i} )^2 \\
	&\le \max\Big(\Delta_{i,i-1} + \Delta_{i+1,i} + \Delta_{j,j-1} + \Delta_{j+1,j}, \Delta_{j,i-1} + \Delta_{i+1,j} + \Delta_{i,j-1} + \Delta_{j+1,i}\Big)^2.
	\end{align*}
	Now each of the terms $\Delta_{i,i+1} + \Delta_{i+1,i}$, $\Delta_{j,j-1} + \Delta_{j+1,j}$ is less than $(n-1)$, and the same holds true for the two other sums. Therefore this yields
	\[
	\pi_n(\abs{f - \IE_{\pi_n} f} \ge t) \le 2 \exp\Big( - \frac{t^2}{8(n-1)^3}\Big).
	\]
	\cite{Cla09} has calculated the variance of the sum of ascents, and it is of order $n^3$, which is in good accordance with the concentration inequality (again, up to the factor).
\end{example}

\begin{example}
	Given a matrix $a = (a_{ij})$ of real numbers satisfying $a_{ij} \in [0,1]$, define $f(\sigma) = \sum_{i = 1}^n a_{i,\sigma(i)}$. By elementary computations one can show $\Gamma(f)^2 \le 4f + 4 \IE_{\pi_n} f$, i.\,e. $f$ is self-bounding. As a consequence, Proposition \ref{proposition:SelfBounding} leads to
	\[
	\pi_n \big( \abs{f - \IE_{\pi_n} f} \ge t \big) \le 2\exp\Big(- \frac{t^2}{32\IE_{\pi_n} f + 8t/3} \Big).
	\]
	
	Concentration inequalities for $f$ have been proven using the exchangeable pair approach in \cite[Proposition 3.10]{Cha05} (see also \cite[Theorem 1.1]{Cha07}), with the denominator being $4\IE_{\pi_n} f + 2t$.
	
	For example, if $a$ is the identity matrix, $f$ is the number of fixed points of a random permutation, which satisfies $\IE_{\pi_n} f = 1$ for all $n \in \IN$. In this case, $f$ converges to a Poisson distribution with mean $1$ as $n \to \infty$ (see e.\,g. \cite{Dia88}).
\end{example}

\begin{example}
	Finally, consider the random variable $f(\sigma) = g(\sigma) + g(\sigma^{-1})$, where $g(\sigma) = \sum_{i = 1}^{n-1} \eins_{\sigma(i+1) > \sigma(i)}$ is the number of descents. In \cite{CD17} the authors calculated the expectation and variance of $f$ and proved a central limit theorem. As in the above example one can easily see that $\Gamma(g)^2 \le 4(n-1)$, as well as $\Gamma(g \circ \mathrm{inv})^2 \le 4(n-1)$, where $\mathrm{inv}: S_n \to S_n$ denotes the inverse map. Since $\Gamma(h_1 + h_2)^2 \le 2\Gamma(h_1)^2 + 2\Gamma(h_2)^2$ holds true for any functions $h_1, h_2$, we also have $\Gamma(f)^2 \le 16(n-1)$, implying for any $t \ge 0$
	\[
	\pi_n(\abs{f - \IE f} \ge t) \le 2\exp\Big(-\frac{t^2}{32(n-1)} \Big).
	\]
	Again, the variance is of order $\sqrt{n}$, so that it is consistent with the CLT.
\end{example}

\subsection{Multilinear polynomials in \texorpdfstring{$[0,1]$}{0,1}-random variables}\label{subsection:MultilinearPolynomials}
The aim of this section is to show Bernstein-type concentration inequalities for a class of polynomials in independent random variables with values in $[0,1]$. The functions we consider are constructed as follows: Let $H = (V, E, (w_e)_{e \in E})$ be a weighted hypergraph, such that every $e \in E$ consists of at most $k$ vertices, assume that $(X_v)_{v \in V}$ are independent, $[0,1]$-valued random variables, and set
\begin{equation}	\label{eqn:Definitionf}
f(X) = f((X_v)_{v \in V}) = \sum_{e \in E} w_e \prod_{v \in e} X_v = \sum_{e \in E} w_e X_e.
\end{equation}
Define the maximum first order partial derivative $\mathrm{ML}(f)$ as
\begin{equation} \label{eqn:MLf}
\mathrm{ML}(f) \coloneqq \sup_{v \in V} \sup_{x \in [0,1]^V} \partial_v f(x).
\end{equation}

\begin{proposition}\label{proposition:HomogeneousPolynomials}
	Let $(X_v)_{v \in V}$ be independent, $[0,1]$-valued random variables and $f: [0,1]^V \to \IR$ given as in \eqref{eqn:Definitionf}. Assume that $w_e \ge 0$ and $\abs{e} \le k$ for all $e \in E$. We have for any $t \ge 0$
	\begin{equation}\label{eqn:Prop11Inequality}
	\IP( f(X) - \IE f(X) \ge t) \le \exp\Big(- \frac{t^2}{2k \mathrm{ML}(f)(\IE f(X) + t/2)} \Big).
	\end{equation}
	Furthermore, for $t \in [0, \IE f]$ it holds
	\[
	\IP(\IE f(X) - f(X) \ge t) \le \exp\Big( - \frac{t^2}{2k\mathrm{ML}(f)} \Big).
	\]
\end{proposition}

A slight modification of the proof of Proposition \ref{proposition:HomogeneousPolynomials} also allows to prove deviation inequalities for suprema of such homogeneous polynomials. For example, this can be used to prove the following concentration inequalities for maxima or $l^p$ norms of linear forms. 

\begin{proposition}\label{proposition:Suprema01ValuedRV}
	Let $(X_v)_{v \in V}$ be independent, $[0,1]$-valued random variables, $\mathcal{F} \subset \{ a \in \IR^V : a_i \in [0,1]^n \}$ and define $f_{\mathcal{F}}(X) \coloneqq \sup_{a \in \mathcal{F}} \sum_{i \in V} a_i X_i$. For any $t \ge 0$ we have
	\[
	\IP(f_{\mathcal{F}}(X) - \IE f_{\mathcal{F}}(X) \ge t) \le \exp\Big( - \frac{t^2}{2\sup_{a \in \mathcal{F}} \norm{a}_\infty(\IE f_{\mathcal{F}}(X) + t/2 )}\Big).
	\]
	In particular, for any $p \in [1,\infty]$ it holds
	\[
	\IP(\norm{X}_p - \IE \norm{X}_p \ge t) \le \exp\Big( - \frac{t^2}{2(\IE \norm{X}_p + t/2)} \Big).
	\]
\end{proposition}

One possible application of Proposition \ref{proposition:HomogeneousPolynomials} is to understand the finite $n$ concentration properties of the so-called \emph{d-runs on the line}.

\begin{proposition}\label{proposition:DRuns}
	Let $d \in \IN, n > d$, $(X_i)_{i = 1,\ldots,n}$ be independent, identically distributed random variables with values in $[0,1]$ and mean $\eta \coloneqq \IE X_1 > 0$. Define the random variable $f_d(X) \coloneqq \sum_{i = 1}^n X_i \cdots X_{i+d-1}$, where the indices are to be understood modulo $n$. For any $t \ge 0$ it holds
	\begin{equation}\label{eqn:ConcentrationDRuns}
	\IP\Big(f_d(X) - \IE f_d(X) \ge \sqrt{n \eta^d} t \Big) \le 2 \exp\Big( - \frac{t^2}{2d^2(1+ t/\sqrt{n\eta^d})} \Big).
	\end{equation}
\end{proposition}

In \cite[Theorem 4.1]{RR09}, the authors prove a CLT for the $d$-runs on the line for Bernoulli random variables $X_i$ with success probability $p$, by normalizing $f$ by $ n^{-1/2} p^{-d/2}$. This is also the reason for the choice $n^{1/2} \eta^{d/2} t$ in inequality \eqref{eqn:ConcentrationDRuns}. In other words, under the assumption $n \eta^d \to \infty$ as $n \to \infty$, Proposition \ref{proposition:DRuns} yields sub-Gaussian tails for $n^{-1/2} \eta^{-d/2}f$. This is in good accordance with the aforementioned CLT.

Moreover, note that in this example, our methodology leads to better results than the usual bounded difference inequality. Indeed, the latter only yields
\[
\IP(\abs{f_d(X) - \IE f_d(X)} \ge t) \le 2 \exp\Big( -\frac{2 t^2}{nd^2}\Big),
\]
suggesting an (inaccurate) normalization of $f(X)$ by $n^{-1/2}$.

\begin{example}
	If $(X_v)_{v \in E(K_n)}$ is the Erd{\"o}s--R{\'e}nyi model with parameter $p$, for any fixed graph $H$ with $\abs{V}$ vertices and $\abs{E}$ edges, the subgraph counting statistic $T_H$ can be written in the form \eqref{eqn:Definitionf} with $w_e = 1$, and $k = \abs{E}$. Furthermore, it is easy to see that $\mathrm{ML}(f) \le n^{\Delta-1}$ for the maximum degree $\Delta$, so that Proposition \ref{proposition:HomogeneousPolynomials} yields
	\[
	\IP(T_H(X) - \IE T_H(X) \ge \epsilon \IE T_H(X)) \le \exp\Big( - C_{k,\epsilon} n^{\abs{V} - \Delta + 1} p^{\abs{E}} \Big).
	\]
	For example, this gives nontrivial bounds in the triangle case whenever $n^2 p^3 \to \infty$ as $n \to \infty$. This bound is suboptimal, as the optimal decay is known to be $np \to \infty$, see \cite{Cha12, DK12a}. However, it is better than the bound obtained by the bounded differences inequality.
	In general, if we consider subgraph counting statistics for some subgraph $H$ with $v$ vertices and $e$ edges on an Erd{\"o}s--R{\'e}nyi model $(X_v)_{v \in E(K_n)}$, the bounded difference inequality yields the estimate
	\[
	\IP(f(X) - \IE f(X) \ge \epsilon \IE f(X)) \le \exp\Big( - C_{\epsilon, H} \frac{n^{2\abs{V}} p^{2\abs{E}}}{n^2 n^{2\Delta -2}} \Big).
	\]
	Thus, to obtain non-trivial estimates in the limit $n \to \infty$, one has to assume that $n^{\abs{V} - \Delta} p^{\abs{E}} \to \infty$. With the above inequality, this can be weakened to $n n^{\abs{V} - \Delta} p^{\abs{E}} \to \infty$.
\end{example}

\subsection{Derivations}\label{subsection:Derivations}
If $\Gamma$ satisfies the chain rule, i.\,e. for all differentiable $u: \IR \to \IR$ and $f \in \mathcal{A}$ such that $u \circ f \in \mathcal{A}$ we have $\Gamma(u \circ f) = \abs{u' \circ f} \Gamma(f)$, then \eqref{eqn:modLSIdef} is equivalent to the usual logarithmic Sobolev inequality (in short: $\Gamma\mathrm{-LSI}(\rho)$)
\[
\mathrm{Ent}_{\mu}(f^2) \le 2\rho \IE_\mu \Gamma(f)^2.
\]
Using this, one can derive second order concentration inequalities similar to the ones given in \cite{BCG17} from Proposition \ref{sq-chr}. Let $S^{n-1} \coloneqq \{x \in \IR^n : \abs{x} = 1 \}$ be the unit sphere equipped with the uniform measure $\sigma_{n-1}$. It is known that for $\rho_n \coloneqq (n-1)^{-1}$
\begin{equation}
\mathrm{Ent}_{\sigma_{n-1}}(e^f) \le \frac{\rho_n}{2} \IE_{\sigma_{n-1}} \abs{\nabla_S f}^2 e^f
\end{equation}
holds for all Lipschitz functions $f$ and the spherical gradient $\nabla_S f$ (see \cite[Formula (3.1)]{BCG17} for the logarithmic Sobolev inequality, from which the modified one follows as above). To state our next result, we introduce the following notation (which we will stick to for the rest of this paper): if $A$ is an $n \times n$ matrix, we denote by $\norm{A}_\mathrm{HS}$ its Hilbert--Schmidt and by $\norm{A}_\mathrm{op}$ its operator norm.

\begin{proposition}
	Consider $S^{n-1}$ equipped with the uniform measure $\sigma_{n-1}$ and let $f: S^{n-1} \to \IR$ be a $C^2$ function satisfying $\sup_{\theta \in S^{n-1}} \norm{f_S''(\theta)}_{\mathrm{op}} \le 1$. For any $t \ge 0$
	\[
	\sigma_{n-1}\big( \abs{ f - \IE_{\sigma_{n-1}} f} \ge t \big) \le 2 \exp\Big( - \frac{1}{4\rho_n} \min\Big( \frac{t^2}{2\IE_{\sigma_{n-1}} \abs{\nabla_S f}^2}, t \Big) \Big).
	\]
\end{proposition}

This follows immediately from Proposition \ref{sq-chr} and the inequality $\abs{\nabla_S \abs{\nabla_S f}} \le \norm{f_S''}_{\mathrm{op}}$ proven in \cite[Lemma 3.1]{BCG17}. Now, if $f$ is $C^2$ and orthogonal to all affine functions (in $L^2(\sigma_{n-1})$), \cite[Proposition 5.1]{BCG17} shows
$
\IE_{\sigma_{n-1}} \abs{\nabla_S f}^2 \le \rho_n \IE_{\sigma_{n-1}} \norm{f''_S}_{\mathrm{HS}}^2.
$
So, if we additionally have $\IE_{\sigma_{n-1}} \norm{f''}_{\mathrm{HS}}^2 \le b^2$, the estimate
\begin{equation}\label{twolevelsphere}
\sigma_{n-1}\big( (n-1) \abs{f - \IE_{\sigma_{n-1}} f} \ge t \big) \le 2 \exp \Big( - \frac{1}{4} \min\Big( \frac{t^2}{2 b^2}, t \Big) \Big)
\end{equation}
follows.

In a similar manner, one may address open subsets of $\mathbb{R}^n$ equipped with some probability measure $\mu$ satisfying a logarithmic Sobolev inequality (with respect to the usual gradient $\nabla$). This situation has been sketched in \cite[Remark 5.3]{BCG17} and was discussed in more detail in \cite{GS16}. Here we easily obtain the following result:

\begin{proposition}
	Let $G \subseteq \mathbb{R}^n$ be an open set, equipped with a probability measure $\mu$ which satisfies a $\nabla\mathrm{-LSI}(\rho)$, and let $f: G \to \IR$ be a $C^2$ function satisfying $\sup_{x \in G} \norm{f''(x)}_{\mathrm{op}} \le 1$. For any $t \ge 0$
	\[
	\mu\big( \abs{ f - \IE_{\mu} f} \ge t \big) \le 2 \exp\Big( - \frac{1}{4\rho} \min\Big( \frac{t^2}{ 2\IE_{\mu} \abs{\nabla f}^2}, t \Big) \Big).
	\]
\end{proposition}

For the proof it only remains to note that $|\nabla|\nabla f|| \le \norm{f''}_{\mathrm{op}}$, cf. \cite[Lemma 7.2]{GS16}. As above, if we require the first order partial derivatives $\partial_i f$ to be centered (which translates into orthogonality to linear functions if $\mu$ is the standard Gaussian measure, for instance), a simple application of the Poincar\'{e} inequality yields
$
\IE_{\mu} \abs{\nabla f}^2 \le \rho \IE_{\mu} \norm{f''}_{\mathrm{HS}}^2.
$
In particular, we have the following corollary which immediately follows from Proposition \ref{sq-chr} and the Poincar{\'e} inequality.

\begin{corollary}\label{corollary:SecondOrderHessian}
	Let $G \subseteq \IR^n$ be an open set, equipped with a probability measure $\mu$ satisfying a $\nabla\mathrm{-LSI}(\rho)$, and $f: G \to \IR$ be a $C^2$ function with
	\[
	\sup_{x \in \mathrm{supp}(\mu)} \norm{f''(x)}_{\mathrm{op}} \le b\quad\text{and}\quad \int \norm{f''(x)}_{\mathrm{HS}}^2 d\mu(x) \le a^2.
	\]
	For any $t \ge 0$ we have
	\[
	\mu\big( \abs{f(x) - \IE_{\mu} f(x) - \skal{x - \IE_\mu (x), \IE_\mu \nabla f(x)}} \ge t \big) \le 2 \exp\Big( - \frac{1}{4} \min\Big( \frac{t^2}{2\rho^2 a^2}, \frac{t}{\rho b} \Big) \Big).
	\]
\end{corollary}

Thus, if we recenter a function and its derivatives, the two conditions on the Hessian ensure two-level concentration inequalities. For functions $f(X,Y)$ of independent Gaussian vectors, two-level concentration inequalities have been studied in \cite{Wo13} using the Hoeffding decomposition instead of a recentering of the partial derivatives.

Note that \eqref{twolevelsphere} and Corollary \ref{corollary:SecondOrderHessian} do not only recover \cite[Theorem 1.1]{BCG17} and \cite[Theorem 1.4]{GS16}, but even strengthen these results by providing two-level bounds. To illustrate this, we discuss one of the examples from \cite{GS16} in more detail.

\begin{example}[Eigenvalues of Wigner matrices]
	Let $\{\xi_{jk}, 1 \le j \le k \le N \}$ be a family of independent real-valued random variables whose distributions all satisfy a $\nabla\mathrm{-LSI}(\rho)$ for a fixed $\rho > 0$. Putting $\xi_{jk} = \xi_{kj}$ for $1 \le k < j \le N$, we define the random matrix $\Xi = (\xi_{jk}/\sqrt{N})_{1 \le j, k \le N}$. Then, by a simple argument using the Hoffman--Wielandt theorem, the joint distribution $\mu^{(N)} = \mu$ of its ordered eigenvalues $\lambda_1 \le \ldots \le \lambda_N$ on $\mathbb{R}^N$ (in fact, $\lambda_1 < \ldots < \lambda_N$ a.s.) satisfies a $\nabla\mathrm{-LSI}(\rho_N)$ with constant $\rho_N = 2 \rho/N$ (see for instance \cite{BG10}).
	
	Now consider a $\mathcal{C}^2$-smooth function $g \colon \mathbb{R}^2 \to \mathbb{R}$ with first order (partial) derivatives in $L^1(\mu)$ and second order derivatives bounded by some constant $\gamma$. Considering a quadratic statistic $\sum_{j \ne k} g(\lambda_j, \lambda_k)$ and recentering according to Corollary \ref{corollary:SecondOrderHessian}, we shall study
	\begin{align*}
	Q_N \coloneqq  &\sum_{j \ne k} g(\lambda_j, \lambda_k) - \sum_{j \ne k}
	\mu[g(\lambda_j, \lambda_k)]\\
	&- \sum_{i=1}^{N}\big(\sum_{k: k \ne i} (\mu[g_x(\lambda_i,
	\lambda_k)]+\mu[g_y(\lambda_k,\lambda_i)])\big)(\lambda_i - \mu[\lambda_i]),
	\end{align*}
	where $g_x, g_y$ denote partial derivatives. For instance, if $g(x,y) \coloneqq xy$, we have $Q_N = \sum_{j \ne k} (\lambda_j - \mu[\lambda_j])(\lambda_k - \mu[\lambda_k])$. Simple calculations show that $\lVert Q_N'' \rVert_{\mathrm{op}} \le c\gamma N$ as well as $\int \lVert Q_N \rVert_{\mathrm{HS}}^2d\mu \le c \gamma^2 N^3$. Here, by $c > 0$ we denote suitable absolute constants which may vary from line to line. Following \cite[Proposition 8.5]{GS16}, this leads to the exponential moment bound
	\[
	\int \exp\left(\frac{c}{\rho \gamma N^{1/2}}\lvert Q_N\rvert\right) d\mu \le 2.
	\]
	By Chebyshev's inequality, $\mu(|Q_N| \ge t) \le 2\exp(-ct/(\rho \gamma N^{1/2}))$ for all $t \ge 0$, thus yielding subexponential fluctuations of order $\mathcal{O}_P(N^{1/2})$.
	
	By contrast, Corollary \ref{corollary:SecondOrderHessian} leads to
	\[
	\mu(\abs{Q_N} \ge t) \le 2 \exp\Big(- c \min\Big( \frac{t^2}{\rho^2 \gamma^2 N}, \frac{t}{\rho \gamma} \Big) \Big),
	\]
	which is much better for large $t$. In particular, the fluctuations in the subexponential regime are of order $\mathcal{O}_P(1)$ now. This can be interpreted as an extension of the self-normalizing property of linear eigenvalue statistics to a second order situation on the level of fluctuations (cf. the discussion of \cite[Proposition 8.5]{GS16}). Note that in \cite{GS16}, a comparable result could be achieved for the special case of $g(x,y) \coloneqq xy$ only.
\end{example}

\subsection{Weakly dependent measures}\label{subsection:WeaklyDependentMeasures}
To continue the discussion of the previous section for a larger class of measures, we will now consider applications of Theorem \ref{theorem:SecondOrderConcentration} for functions of weakly dependent random variables (which, in our case, essentially means that a certain mLSI with respect to a suitable difference operator is satisfied). Throughout this section, we shall consider probability measures $\mu$ on a product of Polish spaces $\mathcal{X} = \otimes_{i = 1}^n \mathcal{X}_i$. For a vector $x = (x_i)_{i \in I}$ and $j \in I$ we let $x_{i^c} = (x_j)_{j \in I \backslash \{i \}}$, and for $y \in \IR$ we write $y_+ = \max(y,0)$. Now we define difference operators on $L^\infty(\mu)$ via
\begin{align*}
\abs{\dpartial f}(x) &= \Big(\sum_{i = 1}^n \int (f(x) - f(x_{i^c}, x_i'))^2 d\mu(x_i' \mid x_{i^c})\Big)^{1/2},\\
\abs{\dpartial^+ f}(x) &= \Big( \sum_{i = 1}^n \int (f(x) - f(x_{i^c}, x_i'))_+^2 d\mu(x_i' \mid x_{i^c}) \Big)^{1/2},\\
\abs{\mathfrak{h} f}(x) &= \Big(\sum_{i = 1}^n \sup_{x_i, x_i'} (f(x) - f(x_{i^c}, x_i'))^2\Big)^{1/2},\\
\abs{\mathfrak{h}^+ f}(x) &= \Big(\sum_{i = 1}^n \sup_{x_i'} (f(x) - f(x_{i^c}, x_i'))_+^2\Big)^{1/2}.
\end{align*}
Here, the suprema over $x_i'$ (and $x_i$) are to be understood with respect to the support of $\mu$. Clearly, $\abs{\dpartial f} \le \abs{\mathfrak{h} f}$ and $\abs{\dpartial^+ f} \le \abs{\mathfrak{h}^+ f}$. Moreover, we need a second order version of the difference operator $\mathfrak{h}$. To this end, for any $i \ne j$, define
\[
\mathfrak{h}_{ij}f(x) = \sup_{x_i, x_i', x_j, x_j'} \abs{f(x) - f(x_{i^c}, x_i' - f(x_{j^c}, x_j') - f(x_{\{i,j\}^c}, x_i',x_j')}
\]
and let $\mathfrak{h}^{(2)}f(x)$ be the matrix (``Hessian'') with zero diagonal and entries $h_{ij}f(x)$ on the off-diagonal.

We now have the following second order result in presence of a $\mathfrak{d}\mathrm{-mLSI}$:

\begin{proposition}\label{proposition:WeaklyDependentRandomVariables}
	Let $\mu$ be a probability measure on a product of Polish spaces $\mathcal{X} = \otimes_{i = 1}^n \mathcal{X}_i$ satisfying a $\mathfrak{d}\mathrm{-mLSI}(\sigma^2)$, and let $f: \mathcal{X} \to \IR$ be a bounded measurable function. If $\abs{\mathfrak{d}^+ \abs{\dpartial f}} \le b$, we have for any $t \ge 0$
	\begin{equation}\label{d-result}
	\mu\big( \abs{f - \IE_\mu f} \ge t \big) \le 2 \exp\Big( -\frac{1}{12\sigma^2} \min\Big( \frac{t^2}{(\IE_\mu \abs{\dpartial f})^2}, \frac{t}{2b} \Big) \Big).
	\end{equation}
	On the other hand, if $\norm{\mathfrak{h}^{(2)} f}_{\mathrm{op}} \le b$ for all $x \in \mathcal{X}$, we have for all $t \ge 0$
	\begin{equation}\label{h-result}
	\mu\big( \abs{f - \IE_\mu f} \ge t \big) \le 2 \exp\Big( -\frac{1}{12\sigma^2} \min\Big( \frac{t^2}{(\IE_\mu \abs{\mathfrak{h} f})^2}, \frac{t}{b} \Big) \Big).
	\end{equation}
\end{proposition}

Proposition \ref{proposition:WeaklyDependentRandomVariables} implies many second order results from previous articles. For instance, it is well-known (and we will check again below) that any product probability measure $\mu$ satisfies a $\mathfrak{d}\mathrm{-mLSI}(1)$. Therefore, from \eqref{d-result} it is easily possible to obtain results similar to \cite[Theorem 1.2]{GS16}. To see this, it suffices to note that for functions with Hoeffding decomposition $f = \sum_{k=2}^n f_k$, one may apply \cite[Proposition 5.2]{GS16} to upper bound $\IE_\mu |\dpartial f|^2$ by $\IE_\mu \lVert \dpartial^{(2)} f \rVert_\mathrm{HS}^2$. Unlike in \cite{GS16}, Proposition \ref{proposition:WeaklyDependentRandomVariables} yields two-level (or Bernstein-type) inequalities, which can be regarded as an advantage of the present approach.

Similarly, we may retrieve (and sharpen) some of the results from further articles like e.\,g. \cite{GSS18} for $d =2$. On the other hand, it seems that requiring modified logarithmic Sobolev inequalities instead of usual logarithmic Sobolev inequalities extends the class of measures to which our results apply, in particular in non-independent situations. We will discuss the $\mathfrak{d}\mathrm{-mLSI}$ property and provide some sufficient conditions in more detail below.

For some classes of functions, we can obtain variants of Proposition \ref{proposition:WeaklyDependentRandomVariables} which are especially adapted to the properties of the functions under consideration. In particular, we may show deviation inequalities for suprema of quadratic forms in the spirit of \cite{KZ18} for the weakly dependent case.

\begin{proposition}	\label{proposition:HansonWright}
	Let $\mu$ be supported in $[-1,+1]^n$ and satisfy a $\mathfrak{d\mathrm{-mLSI}(\sigma^2)}$. Let $\mathcal{A}$ be a countable class of symmetric matrices, bounded in operator norm and with zeroes on its diagonal. Define $h(x) \coloneqq \sup_{A \in \mathcal{A}} \skal{x,Ax}$, $f_{\mathcal{A}}(x) \coloneqq \sup_{A \in \mathcal{A}} \norm{Ax}$ and $\Sigma \coloneqq \sup_{A \in \mathcal{A}} \norm{A}_{\mathrm{op}}$. We have for any $t > 0$
	\begin{equation}
	\mu ( h - \IE_\mu h \ge t ) \le \frac{4}{3} \exp\Big( - \frac{1}{128\sigma^2} \min\Big( \frac{t^2}{2 (\IE_\mu f_{\mathcal{A}})^2}, \frac{t}{\Sigma} \Big) \Big).
	\end{equation}
\end{proposition}

Note that while in general, we only obtain deviation inequalities here, for a single symmetric matrix $A$ with zeroes on its diagonal and the quadratic form $f(x) = \skal{x,Ax}$ similar arguments as in the proof of Proposition \ref{proposition:WeaklyDependentRandomVariables} do lead to concentration inequalities for $f$.

If $\mu$ is a product measure, the result of Proposition \ref{proposition:HansonWright} is well-known and has been proven various times, see for example \cite[Theorem 1.2]{Tal96a} for concentration inequalities in Rademacher random variables, \cite[Theorem 3.1]{Led97} for the upper tail inequalities and random variables satisfying $\abs{X_i} \le 1$, \cite[Theorem 17]{BLM03} for the upper bound and Rademacher random variables and \cite[Corollary 4]{BBLM05}. More recent results include \cite{HKZ12, RV13, Ad15, AKPS18, KZ18, GSS18b}.

To understand which classes of measures may be addressed by Propositions \ref{proposition:WeaklyDependentRandomVariables} and \ref{proposition:HansonWright}, let us study the $\mathfrak{d}\mathrm{-mLSI}$ property in more detail. First, we show that it is implied by another functional inequality. Assume that a probability measure $\mu$ on a product of Polish spaces $\mathcal{X} = \otimes_{i = 1}^n \mathcal{X}_i$ satisfies
\begin{equation}	\label{eqn:modLSI}
\mathrm{Ent}_{\mu}(e^f) \le \sigma^2 \sum_{i = 1}^n \int \mathrm{Cov}_{\mu(\cdot \mid x_{i^c})}(f(x_{i^c}, \cdot), e^{f(x_{i^c},\cdot)}) d\mu(x),
\end{equation}
where $\mu(\cdot \mid x_{i^c})$ denotes the regular conditional probability. This functional inequality is (also) known as a modified logarithmic Sobolev inequality in the framework of Markov processes, and it is equivalent to exponential decay of the relative entropy along the Glauber semigroup, see for example \cite{BT06} or \cite{CMT15}.

\begin{proposition}		\label{proposition:dPlusModLSI}
	If $\mu$ satisfies \eqref{eqn:modLSI}, then a $\dpartial\mathrm{-mLSI}(\sigma^2)$ and a $\dpartial^+\mathrm{-mLSI}(2\sigma^2)$ hold. Consequently, for any $f: \mathcal{X} \to \IR$ and any $\alpha > \sigma^2/2$ we have
	\begin{equation} \label{eqn:exponentialInequalityWithExponent}
	\IE_\mu \exp\Big( f - \IE_\mu f \Big) \le \Big( \IE_\mu \exp \Big( \alpha \abs{\dpartial f}^2 \Big) \Big)^{\frac{\sigma^2}{2\alpha-\sigma^2}}.
	\end{equation}
	The same is true for $\dpartial^+$ with $\sigma^2$ replaced by $2\sigma^2$.
	This especially holds for product measures $\mu = \mu_1 \otimes \ldots \otimes \mu_n$ with $\sigma^2 = 1$.
\end{proposition}

Here, choosing $\alpha = \sigma^2$ or $\alpha = 2\sigma^2$ respectively leads to the exponential inequalities
\[
\IE_\mu \exp(f) \le \IE_\mu \exp\Big( \sigma^2 \abs{\dpartial f}^2 \Big) \quad \text{and} \quad
\IE_\mu \exp(f) \le \IE_\mu \exp\Big( 2 \sigma^2 \abs{\dpartial^+ f}^2 \Big).
\]
The first inequality might be considered as a generalization of \cite[Lemma 8]{Ma00}, which in turn is based on arguments in \cite[Theorem 1.2]{Led97}. The second inequality involving $\abs{\dpartial f}^2$ is well-known in the case of the discrete cube, cf. \cite[Corollary 2.4]{BG99} with a better constant. On the other hand, the proof presented herein is remarkably short and does not rely on some special properties of the measure $\mu$, but can be derived under \eqref{eqn:modLSI}.

Proposition \ref{proposition:dPlusModLSI} implies \cite[Theorem 2]{BLM03}, as product measures satisfy \eqref{eqn:modLSI} with $\sigma^2 = 1$. Indeed, taking the logarithms on both sides of \eqref{eqn:exponentialInequalityWithExponent} gives for any $\alpha > 1$ and $\lambda \ge 0$
\[
\log \IE_\mu \exp\Big( \lambda(f - \IE_\mu f) \Big) \le \frac{1}{\alpha-1} \log \IE_\mu \exp \Big( \lambda^2 \alpha \abs{\dpartial^+ f}^2 \Big).
\]
It remains to choose some fixed $\theta > 0$ and set $\alpha = (\lambda \theta)^{-1}$.

The property \eqref{eqn:modLSI} is satisfied for a large class containing non-product measures. Note that a sufficient condition (due to Jensen's inequality) for \eqref{eqn:modLSI} is the \emph{approximate tensorization property}
\begin{equation}	\label{eqn:AT}
\mathrm{Ent}_{\mu}(e^f) \le \sigma^2 \sum_{i = 1}^n \int \mathrm{Ent}_{\mu(\cdot \mid x_{i^c})}(e^{f(x_{i^c}, \cdot)}) d\mu(x).
\end{equation}
Establishing \eqref{eqn:AT} is subject to ongoing research, and we especially want to highlight two possible approaches.

The first one is akin to the perturbation argument of Holley and Stroock as outlined in \cite{HS87} (see also \cite[Proposition 3.1.18]{Roy07} for a similar reasoning). 
Assume that $d\mu = Z^{-1} e^{f} d\nu$, where $f: \mathcal{X} \to \IR$ is a measurable function, $\nu = \otimes_{i = 1}^n \nu_i$ is \emph{some} product measure and $Z = \IE_\nu e^f$. If we require $f$ to be bounded, we clearly have $\mathrm{osc}(f) < \infty$ for its (maximal) oscillation $\mathrm{osc}(f) = \sup_{x \in \mathcal{X}} f(x) - \inf_{x \in \mathcal{X}} f(x)$. Under these assumptions, $\mu$ satisfies \eqref{eqn:AT} with $\sigma^2 = \exp(2 \mathrm{osc}(f))$. 

Furthermore, under weak dependence conditions  on the local specifications of some measure $\mu$ on a product space $\mathcal{X}$, \eqref{eqn:AT} was proven in \cite{Ma13, Ma15, CMT15}.

\subsection{Bernstein inequality}\label{subsection:Bernstein}
As a final application, let us demonstrate how to recover the classical Bernstein inequality for independent bounded random variables by means of Theorem \ref{theorem:SecondOrderConcentration} (up to constants). In fact, as in some previous works we may remove the boundedness assumption.

There are various extensions of Bernstein's inequality to unbounded random variables. For instance, \cite[Theorem 4]{Ad08} proves deviation inequalities for empirical processes in independent random variables with finite $\Psi_\alpha$ norm for some $\alpha \in (0,1]$, which in particular includes concentration inequalities for sums of random variables with finite $\Psi_\alpha$ norm. Moreover, \cite[Theorem 2.10]{BLM13} requires a certain control of the moments of the random variables, which is in essence a condition on the $\Psi_1$ norms. Thirdly, \cite[Theorem 2.8.1]{Ver18} provides a Bernstein inequality for random variables with bounded $\Psi_1$ norms. However, note that the Gaussian term in the last two mentioned works is a sum of the $\Psi_1$ norm instead of the variance. By our methods, we obtain a version of Bernstein's inequality for sub-Gaussian random variables with the variance of the sum in the Gaussian term, with a reasonable constant.

\begin{theorem}\label{theorem:Bernstein}
	There exists an absolute constant $c' > 0$ such that the following holds. For any set of independent random variables $X_1, \ldots, X_n$ satisfying $\norm{X_i}_{\Psi_2} < \infty$, we have for any $t \ge 0$
	\begin{equation}\label{eqn:BernsteinIneq}
	\IP\big(\abs{\sum_{i = 1}^n X_i - \IE X_i} \ge t\big) \le 2\exp\Big( -\min\Big( \frac{t^2}{80\mathrm{Var}(\sum_i X_i)}, \frac{t}{c'\norm{\max_i \abs{X_i}}_{\Psi_2}} \Big) \Big).
	\end{equation}
	In particular, if $\abs{X_i} \le M$ almost surely for all $i \in \{1,\ldots,n\}$ and some $M > 0$, then for all $t \ge 0$ it holds
	\[
	\IP\big(\abs{\sum_{i = 1}^n X_i - \IE X_i} \ge t\big) \le 2\exp\Big( -\min\Big( \frac{t^2}{80\mathrm{Var}(\sum_i X_i)}, \frac{t}{c'M} \Big) \Big).
	\]
\end{theorem}

We want to give three concluding remarks on Theorem \ref{theorem:Bernstein}.
Firstly, note that is not possible to prove an inequality
\[
\IP\big( \abs{\sum_{i = 1}^n X_i - \IE X_i} \ge t \big) \le 2 \exp\Big( - c\frac{t^2}{\mathrm{Var}(\sum_i X_i)} \Big)
\]
for some absolute constant $c > 0$ in the class of all sub-Gaussian random variables. This can be easily seen in the case $n = 1$ and by choosing $X \sim \mathrm{Ber}(p)$ for $p \to 0$. Thus, to obtain a sub-Gaussian tail with the variance parameter, one has to limit the range of $t$ for which one can expect sub-Gaussian behaviour.

Secondly, one cannot replace $\norm{\max_i \abs{X_i}}_{\Psi_2}$ by $\max_i \norm{X_i}_{\Psi_2}$ in \eqref{eqn:BernsteinIneq}, i.\,e. there cannot be an inequality of the form
\[
\IP\big(\abs{\sum_{i = 1}^n X_i - \IE X_i} \ge t\big) \le 2\exp\Big( -c\min\Big( \frac{t^2}{\mathrm{Var}(\sum_i X_i)}, \frac{t}{\max_i \norm{X_i}_{\Psi_2}} \Big) \Big).
\]
This, again, follows by choosing $X_i \sim \mathrm{Ber}(p)$ for $p = \lambda/n$, $\lambda > 0$. In this case, the sum converges (weakly) to a Poisson random variable, whereas the sub-Gaussian range extends to $t \in \IR_+$ for $n \to \infty$, giving a contradiction.

Thirdly, it is well known that the $\Psi_2$ norm of the maximum of $\Psi_2$ random variables (bounded by some constant, say $K$) grows at most logarithmically in the dimension. For example, if we consider i.\,i.\,d. random variables $X_i$ with unit variance, we have the sub-Gaussian estimate for $t$ of order (at least) $n/\log(n)$.

\section{Proofs and auxiliary results}\label{section:proofs}
We begin by proving Theorem \ref{theorem:SecondOrderConcentration}. Before we start, let us recall \cite[Theorem 2.1]{BG99}, relating the exponential moments of $f - \IE_\mu f$ to those of $\Gamma(f)^2$.

\begin{theorem}\label{theorem:BG99}
	Assume that $(\Omega, \mu, \Gamma)$ satisfies \eqref{eqn:modLSIdef} with constant $\rho > 0$. Then for any $f \in \mathcal{A}$ and any $\alpha > \frac{\rho}{2}$ we have
	\[
	\IE_\mu \exp(f - \IE_\mu f) \le \big( \IE_\mu \exp(\alpha \Gamma(f)^2) \big)^{\frac{\rho}{2\alpha - \rho}}.
	\]
\end{theorem}

Note that formally, Theorem \ref{theorem:BG99} and our own results like Theorem \ref{theorem:SecondOrderConcentration} are valid for bounded functions $f$ only, since $\Gamma$ was defined on a subset of bounded functions. However, it is not hard to see that our proofs can usually be extended to a suitable larger class of functions $\tilde{\mathcal{A}} \supset \mathcal{A}$. One possible approach is first to truncate the random variable $f$ under consideration, and then prove bounds which are independent of the truncation level. As this is somewhat situational and depends on the difference operator $\Gamma$, we stick to the boundedness assumption for the sake of a clearer presentation of the arguments. Nevertheless, we can prove Theorem \ref{theorem:SecondOrderConcentration} under the assumption that $\Gamma$ can be suitably defined for the function $f$ at hand, and that $\Gamma(f) \le g$ for some sub-Gaussian function.

Furthermore, we need an elementary inequality to adjust the constants in concentration or deviation inequalities: for any two constants $c_1 > c_2 > 1$ we have for all $r \ge 0$ and $c > 0$
\begin{equation}\label{eqn:constantAdjustment} 
c_1 \exp(-c r) \le c_2 \exp\Big(-\frac{\log(c_2)}{\log(c_1)} cr\Big)
\end{equation}
whenever the left hand side is smaller or equal to $1$. 

\begin{proof}[Proof of Theorem \ref{theorem:SecondOrderConcentration}]
	Assume that $\rho = 1$, which can always be achieved by defining a new difference operator $\Gamma_{\rho}(f) = \sqrt{\rho}\Gamma(f)$. The general inequality follows by straightforward modifications from the $\rho = 1$ case.
	
	Making use of Theorem \ref{theorem:BG99} in the first and $a^2 \le 2(a-b)_+^2 + 2b^2$ for any $a,b \ge 0$ in the second inequality, we obtain for all $\lambda \ge 0$
	\begin{align*}
	\IE_\mu \exp\big( \lambda (f - \IE_\mu f) \big) &\le \IE_\mu \exp\big( \lambda^2 \Gamma(f)^2 \big) 
	\le \exp\big( 2\lambda^2 c^2 \big) \IE_\mu \exp \big( 2\lambda^2 (g - c)_+^2 \big).
	\end{align*}
	The sub-Gaussian condition \eqref{eqn:gSubGaussian} leads to
	\[
	\int \exp(2\lambda^2 (g-c)_+^2) d\mu \le 1 + \int_0^\infty \exp\Big( -t \Big( \frac{1}{4 \lambda^2 K^2} -1\Big) \Big) dt \le \frac{C}{1 - 4\lambda^2K^2}
	\]
	whenever $4 \lambda^2 K^2 < 1$. Consequently, for all $\lambda \in [0,(2K)^{-1})$ we obtain by Markov's inequality
	\begin{equation}\label{eqn:ExponentialIneq}
	\mu(f - \IE_\mu f \ge t) \le  \frac{C}{1 - 4\lambda^2K^2} \exp\big(-\lambda t + 2 \lambda^2 c^2 \big).
	\end{equation}
	Now we distinguish the two cases $t \le \frac{c^2}{K}$ and $t > \frac{c^2}{K}$. In the first case, set $\lambda \coloneqq \frac{t}{4c^2}$ (which implies $4\lambda^2 K^2 \le 1/4$ and thus is in the range) to obtain
	\begin{align}\label{eqn:Range1}
	\frac{C}{1 - 4\lambda^2 K^2} \exp\big( - \lambda t + 2\lambda^2 c^2 \big) \le \frac{4C}{3} \exp\big( - \frac{t^2}{4c^2} + \frac{t^2}{8c^2} \big) = \frac{4C}{3} \exp\big( - \frac{t^2}{8 c^2} \big),
	\end{align}
	using the monotonicity of $\frac{1}{1-x}$. In the second case, we simply set $\lambda \coloneqq \frac{1}{4 K}$ (implying $\lambda^2 K^2 = 1/4$) and observe that
	\begin{align}\label{eqn:Range2}
	\frac{C}{1 - 4\lambda^2 K^2} \exp\big( - \lambda t + 2\lambda^2 c^2 \big) \le \frac{4C}{3} \exp\big( - \frac{t}{4K} + \frac{c^2}{8K^2} \big) \le \frac{4C}{3} \exp\big( - \frac{t}{8K} \big).
	\end{align}
	Combining \eqref{eqn:Range1} and \eqref{eqn:Range2} finishes the proof of \eqref{eqn:devine}.
	
	Finally, \eqref{eqn:conine} follows by considering $-f$ instead of $f$, which yields
	\[
	\mu(\abs{f - \IE_\mu f} \ge t) \le \frac{8C}{3} \exp\Big( - \frac{1}{8} \min\Big( \frac{t^2}{c^2}, \frac{t}{K} \Big) \Big).
	\]
	The constant can be adjusted using \eqref{eqn:constantAdjustment}.
\end{proof}

\begin{proof}[Proof of Corollary \ref{corollary:concIneqAlteVersion}]
	Using the $\Gamma\mathrm{-mLSI}(\rho)$, by applying Theorem \ref{theorem:BG99} to $f := \lambda g$, Markov's inequality and optimizing it can be shown that for all $t \ge 0$
	\begin{equation}\label{eqn:rechteFlanken}
	\mu( g - \IE_\mu g \ge t ) \le \exp\Big( - \frac{t^2}{2\rho b^2} \Big).
	\end{equation}
	Here, to obtain the factor $2$ in the denominator, one has to let $\alpha \to \infty$ in Theorem \ref{theorem:BG99}. Thus, the corollary follows easily from Theorem \ref{theorem:SecondOrderConcentration}.
\end{proof}

\begin{proof}[Proof of Proposition \ref{sq-chr}]
	We assume $b = 1$ which can be done by rescaling.
	
	First, observe that \cite[equation (2.4)]{BG99} holds for any positive function $g$, since the inequality $\Gamma(g^2) \le 2 g \Gamma(g)$ is sufficient to apply the argument given therein. Thus, for any positive function $g$ satisfying $\Gamma(g) \le 1$ it holds for $\lambda \in [0,(2\rho)^{-1})$
	\begin{equation}\label{(2.4)}
	\IE_\mu \exp(\lambda g^2) \le \exp\Big( \frac{\lambda}{1-2\rho \lambda} \IE_\mu g^2 \Big).
	\end{equation}
	So, by applying Theorem \ref{theorem:BG99} (with $\alpha = \rho)$ we have
	\[
	\IE_\mu \exp\Big(\frac{1}{2\rho}(f - \IE_\mu f) \Big) \le \IE_\mu \exp\Big( \frac{1}{4\rho} \Gamma(f)^2 \Big) \le \IE_\mu \exp\Big( \frac{1}{4\rho} g^2 \Big) \le \exp\Big( \frac{1}{2\rho} \IE_\mu g^2 \Big),
	\]
	which can also be applied to $\lambda f$ and $\lambda g$ instead of $f$ and $g$, for $\lambda \in [0,1]$. Thus, by Markov's inequality, for any $\lambda \in [0,1]$
	\[
	\mu(f - \IE_\mu f \ge t) \le \exp\Big( - \frac{\lambda t}{2\rho} + \frac{\lambda^2}{2\rho} \IE_\mu g^2 \Big).
	\]
	The claim follows by putting $\lambda = \min(\frac{t}{2 \IE_\mu g^2}, 1)$ and noting that if $t/(2\IE_\mu g^2) \ge 1$, we have $t - \IE_\mu g^2 \ge t/2$.
\end{proof}

\begin{proof}[Proof of Proposition \ref{proposition:SelfBounding}]
	Choosing $\alpha = \rho$ in Theorem \ref{theorem:BG99}, applying the inequality to $\lambda f$ and using the monotonicity leads to
	\begin{align*}
	\IE_{\mu} \exp\big( \lambda (f - \IE_\mu f) \big) \le \exp\big( \lambda^2 \rho (b + a\IE_\mu f) \big) \IE_{\mu} \exp\big( \lambda^2 \rho a(f - \IE_\mu f) \big).
	\end{align*}
	Thus for $\lambda \in (0,(a\rho)^{-1})$, by Jensen's inequality (applied to the concave function $x \mapsto x^{\lambda \rho a}$) we have
	\[
	\big( 1 - \lambda \rho a \big) \log\big( \IE_\mu \exp\big( \lambda(f - \IE_\mu f) \big) \big) \le \lambda^2 \rho(b + a\IE_\mu f).
	\]
	Finally, Markov's inequality and \cite[Lemma 11]{BLM03} yield the first inequality.
	
	To see the second inequality, note that for any $\lambda > 0$ such that $\lambda a \rho < 1$, by Theorem \ref{theorem:BG99} and concavity of $x \mapsto x^{\lambda a \rho}$, it holds
	\begin{align*}
	\IE_{\mu} \exp\big(\lambda(\IE_{\mu}f - f) \big) &\le \IE_\mu \exp\big( \rho \Gamma(-\lambda f)^2 \big) = \IE_\mu \exp \big( \lambda^2 \rho \Gamma(f)^2 \big) \\
	&\le \IE_\mu \exp\big( \lambda^2 \rho (af+b) \big) \\
	&= \exp\big(\lambda^2 \rho (a \IE_\mu f +b) \big) \IE_\mu \exp\big(\lambda^2 \rho a(f - \IE_\mu f) \big) \\
	&\le \exp\big(\lambda^2 \rho(a \IE_\mu f +b) \big) \Big( \IE_\mu \exp\big( \lambda(f - \IE_\mu f) \big) \Big)^{\lambda \rho a}.
	\end{align*}
	Finally, applying the estimates from the first part we obtain
	\[
	\IE_\mu \exp\big( \lambda(\IE_{\mu} f - f) \big) \le \exp\Big( \frac{\lambda^2\rho }{1-\lambda \rho a}(a \IE_\mu f + b) \Big).
	\]
	The concentration inequality follows as in the first part.
\end{proof}

\begin{proof}[Proof of Proposition \ref{proposition:mLSIforSymmetricGroup}]
	Using and rewriting \cite[Theorem 1]{GQ03} we obtain for any $f: S_n \to \IR$
	\[
	\mathrm{Ent}(e^f) \le \frac{1}{2n n!} \sum_{i,j = 1}^n \sum_{\sigma \in S_n} (f(\sigma \tau_{ij}) - f(\sigma))(e^{f(\sigma \tau_{ij})} - e^{f(\sigma)}).
	\]
	Now, the inequality $(a-b)(e^a - e^b) \le \frac{1}{2} (e^a + e^b)(a-b)^2$ and the fact that $\sigma \mapsto \sigma \tau_{ij}$ is an automorphism of $S_n$ leads to the $\Gamma\mathrm{-mLSI}(1)$. The $\Gamma^+\mathrm{-mLSI}(2)$ follows in the same manner from the inequality $(a-b)_+(e^a - e^b) \le (a-b)_+^2 e^a$.
\end{proof}

\begin{proof}[Proof of Theorem \ref{theorem:LocallyLipschitzSn}]
	By Proposition \ref{proposition:mLSIforSymmetricGroup} and Theorem \ref{theorem:BG99} we have for any $f: S_n \to \IR$, any $\lambda \in \IR$ and any $\alpha > 1/2$ the inequality
	\[
	\IE_{\pi_n} \exp\big( \lambda(f - \IE_{\pi_n} f) \big) \le \Big(\IE_{\pi_n} \exp\big( \alpha \lambda^2 \Gamma(f)^2 \big) \Big)^{\frac{1}{2\alpha-1}}.
	\]
	If $f$ is locally Lipschitz with respect to $d$, an easy calculation shows that we can upper bound $\Gamma(f)^2 \le \mathrm{ObsDiam}(S_n,d)$, so that from the above inequality in combination with $\alpha \to \infty$ we get
	\[
	\IE_{\pi_n} \exp\big( \lambda(f - \IE_{\pi_n} f) \big) \le \exp\big( \lambda^2\mathrm{ObsDiam}(S_n,d) /2 \big).
	\]
	The sub-Gaussian estimate follows by Markov's inequality and the variance bound from integration by parts.
\end{proof}

In order to prove Proposition \ref{Tcde}, we first need to establish the following lemma:

\begin{lemma}\label{lemfTcde}
	Let $f: S_n \to \IR$ be a non-negative function such that
	\begin{enumerate}
		\item $\Gamma^+(f)^2 \le f$,
		\item $\abs{f(\sigma) - f(\sigma \tau_{ij})} \le 1$ for all $\sigma, i,j$.
	\end{enumerate}
	Then for all $t \in [0, \IE_{\pi_n} f]$ we have
	\[
	\pi_n(\IE_{\pi_n} f - f \ge t) \le \exp\Big( -\frac{t^2}{8\IE_{\pi_n} f}. \Big)
	\]
	Especially we have
	\[
	\pi_n(f = 0)\exp\Big( \frac{\IE_{\pi_n} f}{8} \Big) \le 1.
	\]
	In particular, this holds for $f(\sigma) = \frac{1}{16} d_T(\sigma, A)^2$, where $A \subset S_n$ is any set.
\end{lemma}

\begin{proof}[Proof of Lemma \ref{lemfTcde}]
	Rewriting \cite[Theorem 1]{GQ03}, we have that for any positive function $g$,
	\begin{align*}
	\mathrm{Ent}_{\pi_n}(g)
	&\le \frac{1}{n!} \frac{1}{2n} \sum_{i,j} \sum_{\sigma \in S_n} (g(\sigma\tau_{ij}) - g(\sigma))(\log g(\sigma \tau_{ij}) - \log g(\sigma)) \\
	&= \frac{1}{n!} \frac{1}{n} \sum_{i,j} \sum_{\sigma \in S_n} (g(\sigma \tau_{ij}) - g(\sigma)) (\log g(\sigma\tau_{ij}) - \log g(\sigma))_+.
	\end{align*}
	Using this, we obtain for any $\lambda \in [0,1]$
	\begin{align*}
	\mathrm{Ent}_{\pi_n}(e^{-\lambda f}) &\le \frac{\lambda}{n} \IE_{\pi_n} \sum_{i,j} (f(\sigma) - f(\sigma \tau_{ij}))_+ \big( \exp(-\lambda f(\sigma \tau_{ij})) - \exp(-\lambda f(\sigma)) \big) \\
	&\le \frac{\lambda}{n} \IE_{\pi_n} \sum_{i,j} (f(\sigma) - f(\sigma \tau_{ij}))_+ (\exp(\lambda(f(\sigma) - f(\sigma \tau_{ij})))-1) e^{-\lambda f(\sigma)} \\
	&\le \frac{\lambda}{n} \IE_{\pi_n} \sum_{i,j} (f(\sigma) - f(\sigma \tau_{ij}))_+ \Psi(\lambda(f(\sigma) - f(\sigma \tau_{ij}))) e^{-\lambda f(\sigma)},
	\end{align*}
	where $\Psi(x) \coloneqq e^x - 1$. By a Taylor expansion it can easily be seen that $\Psi(x) \le 2x$ for all $x \in [0,1]$, so that (recall that by $(2)$ we have $f(\sigma) - f(\sigma \tau_{ij}) \le 1$, and $f(\sigma) - f(\sigma \tau_{ij}) \ge 0$ due to the positive part)
	\begin{align*}
	\mathrm{Ent}_{\pi_n}(e^{-\lambda f}) &\le \frac{2\lambda^2}{n}  \IE_{\pi_n} \sum_{i,j} (f(\sigma) - f(\sigma \tau_{ij}))_+^2 e^{-\lambda f(\sigma)}\\ &= 2\lambda^2 \IE_{\pi_n} \Gamma^+(f)^2 e^{-\lambda f} \le 2\lambda^2 \IE_{\pi_n} f e^{-\lambda f}.
	\end{align*}
	Chebyshev's association inequality yields
	\[
	\mathrm{Ent}_{\pi_n}(e^{-\lambda f}) \le 2\lambda^2 \IE_{\pi_n} f \IE_{\pi_n} e^{-\lambda f}.
	\]
	In other terms, if we set $h(\lambda) \coloneqq \IE_{\pi_n} e^{-\lambda f}$, we have
	\[
	\Big( \frac{\log h(\lambda)}{\lambda} \Big)' \le 2\IE_{\pi_n} f,
	\]
	which by the fundamental theorem of calculus implies for all $\lambda \in [0,1]$
	\[
	\IE_{\pi_n} \exp\Big( \lambda(\IE_{\pi_n} f - f) \Big) \le \exp\Big( 2\lambda^2 \IE_{\pi_n} f \Big).
	\]
	So, for any $t \in [0, \IE_{\pi_n} f]$, by Markov's inequality and setting $\lambda = \frac{t}{4 \IE_{\pi_n} f}$
	\[
	\pi_n(\IE_{\pi_n} f - f \ge t) \le \exp\Big(-\lambda t + 2\lambda^2 \IE_{\pi_n} f\Big) = \exp\Big( - \frac{t^2}{8\IE_{\pi_n} f} \Big).
	\]
	The second part follows by nonnegativity and $t = \IE_{\pi_n} f$.
	
	It remains to show that $f(\sigma) = \frac{1}{16} d_T(\sigma, A)^2$ satisfies the two conditions of this lemma. To this end, we first need to show that $\Gamma^+(d_T(\cdot,A))^2 \le 4$. Writing $g(\sigma) \coloneqq d_T(\sigma, A)$, it is well known (see \cite{BLM03}) that we have
	\begin{equation}\label{Sion}
	g(\sigma) = \inf_{\nu \in \mathcal{M}(A)} \sup_{\alpha \in \IR^n : \abs{\alpha}_2 = 1} \sum_{k = 1}^n \alpha_k \nu(\sigma' : \sigma'_k \neq \sigma_k),
	\end{equation}
	where $\mathcal{M}(A)$ is the set of all probability measures on $A$. To estimate $\Gamma^+(g)^2(\sigma)$, one has to compare $g(\sigma)$ and $g(\sigma \tau_{ij})$. To this end, for any $\sigma \in S_n$ fixed, let $\tilde{\alpha}, \tilde{\nu}$ be parameters for which the value $g(\sigma)$ is attained, and let $\hat{\nu} = \hat{\nu}_{ij}$ be a minimizer of $\inf_{\nu \in \mathcal{M}(A)} \sum_{k = 1}^n \tilde{\alpha}_k \nu(\sigma' : \sigma'_k \neq (\sigma \tau_{ij})_k)$. This leads to
	\begin{align*}
	\Gamma^+ (g)(\sigma)^2 
	&\le \frac{1}{n} \sum_{i,j = 1}^n \Big( \sum_{k = 1}^n \tilde{\alpha}_k (\hat{\nu}(\sigma'_k \neq \sigma_k) - \hat{\nu}(\sigma'_k \neq (\sigma \tau_{ij})_k))\Big)_+^2 \\
	&\le \frac{2}{n} \sum_{i,j = 1}^n (\tilde{\alpha}_i^2 + \tilde{\alpha}_j^2)
	\le 4.
	\end{align*}
	Using this and the non-negativity of $d_T(\cdot, A)$, we have 
	\[
	\Gamma^+(f)^2 = \frac{1}{256} \Gamma^+(d_T(\cdot, A)^2)^2 \le \frac{1}{64} d_T(\cdot,A)^2 \Gamma^+(d_T(\cdot,A))^2 \le f.
	\]
	
	To show the second property, we proceed similarly to \cite[Proof of Lemma 1]{BLM09}. By \eqref{Sion} and the Cauchy--Schwarz inequality, we have
	\[
	f(\sigma) = \frac{1}{16} \inf_{\nu \in \mathcal{M}(A)} \sum_{k = 1}^n \nu(\sigma' : \sigma'_k \neq \sigma_k)^2.
	\]
	Assuming without loss of generality that $f(\sigma) \ge f(\sigma \tau_{ij})$, choose $\hat{\nu} = \hat{\nu}_{ij} \in \mathcal{M}(A)$ such that the value of $f(\sigma \tau_{ij})$ is attained. It follows that
	\begin{equation*}
	f(\sigma) - f(\sigma \tau_{ij}) \le \frac{1}{16}\sum_{k=1}^n \hat{\nu}(\sigma'_k \neq \sigma_k)^2 - \hat{\nu} (\sigma'_k \neq (\sigma\tau_{ij})_k)^2 \le \frac{2}{16},
	\end{equation*}
	which finishes the proof.
\end{proof}

The proof of Proposition \ref{Tcde} is now easily completed:

\begin{proof}[Proof of Proposition \ref{Tcde}]
	The difference operator $\Gamma^+$ satisfies $\Gamma^+(g^2) \le 2g \Gamma^+(g)$ for all positive functions $g$, as well as an $\mathrm{mLSI}(2)$. Moreover, as seen in the proof of Lemma \ref{lemfTcde}, we have $\Gamma^+(d_T(\cdot, A)) \le 2$. Thus, by \eqref{(2.4)} it holds for $\lambda \in [0,1/4)$
	\[
	\pi_n(A) \IE_{\pi_n} \exp\Big( \frac{\lambda}{4} d_T(\cdot,A)^2 \Big) \le \pi_n(A) \exp\Big( \frac{\lambda}{4-16\lambda} \IE_{\pi_n} d_T(\cdot,A)^2 \Big).
	\]
	Furthermore, Lemma \ref{lemfTcde} shows that
	\[
	\pi_n(A) \exp\Big( \frac{\IE_{\pi_n} d_T(\cdot,A)^2}{128} \Big) \le 1.
	\]
	So, for $\lambda = 1/36$ we have
	\[
	\pi_n(A) \IE_{\pi_n} \exp\Big( \frac{d_T(\cdot,A)^2}{144} \Big) \le \pi_n(A) \exp\Big( \frac{1}{128} \IE_{\pi_n} d_T(\cdot,A)^2 \Big) \le 1.
	\]
\end{proof}

\begin{proof}[Proof of Proposition \ref{proposition:Talagrand}]
	Again, the proof mimics the proof given for independent random variables in \cite{BLM03}. As stated in Proposition \ref{proposition:mLSIforSymmetricGroup}, the uniform measure $\pi_n$ on $S_n$ satisfies a $\Gamma^+\mathrm{-mLSI}(2)$ with respect to
	\[
	\Gamma^+(f)(\sigma)^2 = \frac{1}{n} \sum_{i,j = 1}^n (f(\sigma) - f(\sigma \tau_{ij}))_+^2.
	\]
	Writing $f_A(\sigma) \coloneqq d_T(\sigma, A)$, we have $\Gamma^+ (f_A)(\sigma)^2 \le 4$ as seen in the proof of Lemma \ref{lemfTcde}.
	Hence, by similar arguments as in the proof of Theorem \ref{theorem:SecondOrderConcentration} we have for any $\lambda \ge 0$
	\begin{align}\label{eqn:ExponentialIneqfA}
	\begin{split}
	\IE_{\pi_n} \exp\big( \lambda(f_A - \IE_{\pi_n} f_A) \big) &\le \exp( 4\lambda^2),
	\end{split}
	\end{align}
	implying the sub-Gaussian estimate
	$
	\pi_n(f_A - \IE_{\pi_n} f_A \ge t) \le \exp(- t^2/16).
	$
	Fix a set $A \subseteq S_n$ satisfying $\pi_n(A) \ge 1/2$. As a $\Gamma\mathrm{-mLSI}(1)$ implies a Poincar{\'e} inequality (see \cite[Proposition 3.5]{BT06} or \cite{DSC96}), we also have (by Chebyshev's inequality)
	\[
	t^2 \pi_n\big( f_A - \IE_{\pi_n} f_A \le -t \big) \le \mathrm{Var}_{\pi_n}(f_A) \le 2 \IE_{\pi_n} \Gamma^+(f_A)^2 \le 8,
	\]
	which evaluated at $t = \IE_{\pi_n} f_A$ yields $(\IE_{\pi_n} f_A)^2 \le 16$. Thus, for any $t \ge 4$ it holds
	\begin{equation} \label{eqn:SnTalagrand}
	\pi_n(f_A \ge t) \le \exp\Big( - \frac{(t-4)^2}{16} \Big) \le 2 \exp\Big( - \frac{t^2}{64} \Big),
	\end{equation}
	where the last inequality follows from $(t-4)^2 \ge t^2/2 - 16$ for any $t \ge 0$ and \eqref{eqn:constantAdjustment}. For $t \le 4$ the inequality \eqref{eqn:SnTalagrand} holds trivially.
\end{proof}

The proofs of the results for slices of the hypercube work in a very similar way.

\begin{proof}[Proof of Proposition \ref{proposition:mLSIforBernoulliLaplace}]
	It follows from \cite[Theorem 1]{GQ03} that we have for any $f: C_{n,r} \to \IR$
	\[
	\mathrm{Ent}(e^f) \le \frac{1}{n \binom{n}{r}} \sum_{\eta \in C_{n,r}} \sum_{i < j} (f(\tau_{ij}\eta) - f(\eta))(e^{f(\tau_{ij}\eta)} - e^{f(\eta)}).
	\]
	From here, we may process as in the proof of Proposition \ref{proposition:mLSIforSymmetricGroup}.
\end{proof}

For the proof of Proposition \ref{TcdeBL}, we need to establish the following analogue of Lemma \ref{lemfTcde}:

\begin{lemma}\label{lemfTcdeBL}
	Let $f: C_{n,r} \to \IR$ be a non-negative function such that
	\begin{enumerate}
		\item $\Gamma^+(f)^2 \le f$,
		\item $\abs{f(\eta) - f(\tau_{ij}\eta)} \le 1$ for all $\eta, i,j$.
	\end{enumerate}
	Then for all $t \in [0, \IE_{\mu_{n,r}} f]$ we have
	\[
	\mu_{n,r}(\IE_{\mu_{n,r}} f - f \ge t) \le \exp\Big( -\frac{t^2}{8\IE_{\mu_{n,r}} f}. \Big)
	\]
	Especially we have
	\[
	\mu_{n,r}(f = 0)\exp\Big( \frac{\IE_{\mu_{n,r}} f}{8} \Big) \le 1.
	\]
	In particular, this holds for $f(\eta) = \frac{1}{32} d_T(\eta, A)^2$, where $A \subset C_{n,r}$ is any set.
\end{lemma}

\begin{proof}[Proof of Lemma \ref{lemfTcdeBL}]
	Rewriting \cite[Theorem 1]{GQ03}, we have that for any positive function $g$,
	\begin{align*}
	\mathrm{Ent}_{\mu_{n,r}}(g)
	&\le \frac{1}{\binom{n}{r}} \frac{1}{n} \sum_{i<j} \sum_{\eta \in C_{n,r}} (g(\tau_{ij}\eta) - g(\eta))(\log g(\tau_{ij}\eta) - \log g(\eta)) \\
	&= \frac{1}{\binom{n}{r}} \frac{2}{n} \sum_{i<j} \sum_{\eta \in C_{n,r}} (g(\tau_{ij}\eta) - g(\eta)) (\log g(\tau_{ij}\eta) - \log g(\eta))_+.
	\end{align*}
	From here, we may mimic the proof of Lemma \ref{lemfTcde}.
	
	Last, we need to show that $f(\eta) = \frac{1}{32} d_T(\eta, A)^2$ satisfies the two conditions of this lemma. As compared to the proof of Lemma \ref{lemfTcde}, some of the constants will change because of the different normalization of the difference operators. However, we may argue similarly and show that $\Gamma^+(d_T(\cdot,A))^2 \le 8$. Using this and the non-negativity of $d_T(\cdot, A)$ yields
	\[
	\Gamma^+(f)^2 = \frac{1}{1024} \Gamma^+(d_T(\cdot, A)^2)^2 \le \frac{1}{256} d_T(\cdot,A)^2 \Gamma^+(d_T(\cdot,A))^2 \le f.
	\]
	Finally, by arguing as above it is easily seen that $\abs{f(\eta) - f(\tau_{ij}\eta)} \le 2/32$.
\end{proof}

\begin{proof}[Proof of Proposition \ref{TcdeBL}]
	As the difference operator $\Gamma^+$ satisfies $\Gamma^+(g^2) \le 2g \Gamma^+(g)$ for all positive functions $g$, as well as an $\mathrm{mLSI}(2)$, it remains to change the proof of Proposition \ref{Tcde} in view of the different constants appearing in Lemma \ref{lemfTcdeBL}. As noted in the proof of Lemma \ref{lemfTcdeBL}, we have $\Gamma^+(d_T(\cdot, A)) \le \sqrt{8}$. Thus, by \eqref{(2.4)} it holds for $\lambda \in [0,1/4)$
	\[
	\mu_{n,r}(A) \IE_{\mu_{n,r}} \exp\Big( \frac{\lambda}{8} d_T(\cdot,A)^2 \Big) \le \mu_{n,r}(A) \exp\Big( \frac{\lambda}{8-32\lambda} \IE_{\mu_{n,r}} d_T(\cdot,A)^2 \Big).
	\]
	Furthermore, Lemma \ref{lemfTcdeBL} shows that
	\[
	\mu_{n,r}(A) \exp\Big( \frac{\IE_{\mu_{n,r}} d_T(\cdot,A)^2}{256} \Big) \le 1.
	\]
	So, for $\lambda = 1/68$ we have
	\[
	\mu_{n,r}(A) \IE_{\mu_{n,r}} \exp\Big( \frac{d_T(\cdot,A)^2}{544} \Big) \le \mu_{n,r}(A) \exp\Big( \frac{1}{256} \IE_{\mu_{n,r}} d_T(\cdot,A)^2 \Big) \le 1.
	\]
\end{proof}

Finally, we present the proofs of Section \ref{section:Applications}.

\begin{proof}[Proof of Proposition \ref{proposition:HomogeneousPolynomials}]
	We show that $f$ is weakly $(k\mathrm{ML}(f),0)$-self bounding in the language of \cite{BLM09}. To see this, for any $v \in V$ let $f_v(x_{v^c}) \coloneqq \sum_{e \in E : v \notin E} w_e X_e = f(X_{v^c}, 0)$. Now we have
	\begin{align*}
	\sum_{v \in V} (f(x) - f_v(x_{v^c}))^2 &= \sum_{v \in V} \Big( X_v \sum_{e \in E : v \in e} w_e X_{e \backslash v} \Big)^2 \le \sum_{v \in V} X_v \partial_v f(X)^2 \\
	&\le \mathrm{ML}(f) \sum_{v \in V} X_v \partial_v f(X) \le k \mathrm{ML}(f) f(X).
	\end{align*}
	Here, the first inequality follows from $X_v \in [0,1]$ and the last one is a consequence of Euler's homogeneous function theorem and the fact that all quantities involved are positive. Consequently, \cite[Theorem 1]{BLM09} yields for any $t \ge 0$
	\[
	\IP(f(X) - \IE f(X) \ge t) \le \exp\Big( - \frac{t^2}{2k\mathrm{ML}(f)(\IE f(X) + t/2)} \Big).
	\]
	For the lower bound, apply \cite[Theorem 1]{BLM09} to $\tilde{f} = \mathrm{ML}(f)^{-1} f$ which satisfies $0 \le \tilde{f}(x) - \tilde{f}_v(x_{v^c}) \le 1$  for all $v \in V$ and $x \in [0,1]^V$ and is weakly $(k\mathrm{ML}(f)^{-1},0)$-self bounding.
\end{proof}

\begin{proof}[Proof of Proposition \ref{proposition:Suprema01ValuedRV}]
	The first part follows as above. As for the second part, if we choose $\mathcal{F} = \mathcal{F}_q = \{ a \in \IR^V : a_v \ge 0, \norm{a}_q \le 1 \}$ for some $q \in [1,\infty]$ this leads to 
	\[
	f_{\mathcal{F}}(X) = \sup_{a \in \mathcal{F}_q} \sum_{v \in V} a_v X_v = \Big( \sum_{v \in V} \abs{X_v}^p \Big)^{1/p}
	\] 
	for the H{\"o}lder conjugate $p$, which is due to the nonnegativity of the $X_i$ and the dual formulation of the $L^p$ norm in $\IR^V$.
\end{proof}

\begin{proof}[Proof of Proposition \ref{proposition:DRuns}]
	Clearly, $f_d$ is $d$-homogeneous and has positive weights in the sense of \eqref{eqn:Definitionf}, if we set $V = [n]$ and $E = \{ \{j, j+1, \ldots, j+d-1 \}, j = 1,\ldots, n\}$, $w_e = 1$. Furthermore, the partial derivatives can be easily bounded: For any fixed $l \in [n]$ there are exactly $d$ terms which depend on $X_l$, and the product is bounded by $1$. Consequently, $\mathrm{ML}(f_d) = \max_{l \in [n]} \max_{x \in [0,1]^n} \partial_l f(X) = d.$ Thus, Proposition \ref{proposition:HomogeneousPolynomials} yields for all $t \ge 0$
	\[
	\IP(f_d(X) - \IE f_d(X) \ge t) \le \exp\Big( - \frac{t^2}{2d^2(\IE f_d(X) + t/2)} \Big).
	\]
	The assertion now follows, if we note that $\IE f_d(X) = n \eta^d$.
\end{proof}

Let us now prove the results from Section \ref{subsection:WeaklyDependentMeasures}. To this end, we first need to establish some basic properties of modified logarithmic Sobolev inequalities with respect to the difference operators we use.

\begin{lemma}\label{mLSIs}
	Let $\mu$ be a probability measure on a product of Polish spaces $\mathcal{X} = \otimes_{i = 1}^n \mathcal{X}_i$ which satisfies a $\mathfrak{d}\mathrm{-mLSI}(\sigma^2)$. Then, $\mu$ also satisfies a $\mathfrak{d}^+\mathrm{-mLSI}(2\sigma^2)$.
\end{lemma}

\begin{proof}
	Let $(\Omega, \mathcal{F}, \nu)$ be a probability space and $g$ a measurable function on it. Then,
	\begin{align*}
	&\iint (g(x) - g(y))^2 d\nu(y) e^{g(x)} d\nu(x)\\
	= \ &\iint \big((g(x) - g(y))_+^2e^{g(x)} + (g(y) - g(x))_+^2e^{g(x)}\big) d\nu(y) d\nu(x)\\
	\le \ &\iint \big((g(x) - g(y))_+^2e^{g(x)} + (g(y) - g(x))_+^2e^{g(y)}\big) d\nu(y) d\nu(x)\\
	= \ &2 \iint (g(x) - g(y))_+^2 d\nu(y) e^{g(x)} d\nu(x).
	\end{align*}
	Applying this to $\nu = \mu(\cdot \mid x_{i^c})$ and $g = f(x_{i^c}, \cdot)$ for any $i = 1, \ldots, n$ yields
	\[
	\int \abs{\dpartial f}^2 e^f d\mu \le 2 \int \abs{\dpartial^+ f}^2 e^f d\mu,
	\]
	which finishes the proof.
\end{proof}

Also note that by monotonicity a $\mathfrak{d}-\mathrm{mLSI}(\sigma^2)$ implies an $\mathfrak{h}-\mathrm{mLSI}(\sigma^2)$, and the same holds for $\mathfrak{d}^+$ and $\mathfrak{h}^+$. Moreover, we recall the duality formula $\abs{x} = \sup_{y \in S^{n-1}} \skal{x,y}$.

\begin{proof}[Proof of Proposition \ref{proposition:WeaklyDependentRandomVariables}]
	First, \eqref{d-result} follows by applying Theorem \ref{theorem:SecondOrderConcentration} to $g = \abs{\dpartial f}$ and noting that $\abs{\dpartial(af)} = \abs{a} \abs{\dpartial f}$ for all $a \in \IR$. To see that $g$ is sub-Gaussian with parameter $K = \sqrt{2\sigma^2}b$ and $C=1$, note that by Lemma \ref{mLSIs}, $\mu$ satisfies a $\dpartial^+\mathrm{-mLSI}(2\sigma^2)$, so that we can use \eqref{eqn:rechteFlanken}.
	
	The same arguments are valid for $\mathfrak{h}^+$ and $\mathfrak{h}$ respectively. Here, we additionally use the estimate $\abs{\mathfrak{h}^+ \abs{\mathfrak{h} f}} \le \abs{\mathfrak{h}^{(2)}f}_{\mathrm{op}}$ (cf. \cite[Lemma 3.2]{GSS18b}).
\end{proof}

\begin{proof}[Proof of Proposition \ref{proposition:HansonWright}]
	Let us bound $\abs{\dpartial^+ h}^2$. Choose the matrix $\tilde{A} \in \mathcal{A}$ maximizing $\sup_{A \in \mathcal{A}} \skal{x,Ax}$ and use the monotonicity of $y \mapsto y_+$ to obtain
	\begin{align*}
	\abs{\dpartial^+ h(x)}^2 &= \sum_{i = 1}^n \int (g(x) - g(x_{i^c}, x_i'))_+^2 d\mu(x_i' \mid x_{i^c}) \le \sum_{i = 1}^n \sup_{x_i'} \Big( 2(x_i - x_i') \sum_{j = 1}^n \tilde{A}_{ij} x_j \Big)_+^2 \\
	&\le 16 \norm{\tilde{A} x}_2^2 \le16 \sup_{A \in \mathcal{A}} \norm{Ax}_2^2 = 16 f_{\mathcal{A}}^2(x).
	\end{align*}
	Furthermore, we have for some maximizer $\tilde{A} \in \mathcal{A}$ of $\sup_{A \in \mathcal{A}} \norm{Ax}$ and $\tilde{v} \in S^{n-1}$ for $\sup_{v \in S^{n-1}} \skal{\tilde{A}x, v}$
	\begin{align*}
	\abs{\dpartial^+ f_{\mathcal{A}}}^2 &\le \sum_i \sup_{x_i'} \big( \sup_{v} \skal{\tilde{A}x,v} - \sup_v \skal{\tilde{A}(x_{i^c}, x_i'), v)} \big)_+^2 \le \sum_i \sup_{x_i'} \big( (x_i - x_i') \skal{\tilde{A}e_i, \tilde{v}} \big)_+^2 \\
	&\le 4 \sum_i \skal{\tilde{A}e_i, \tilde{v}}^2 \le 4 \big( \sup_{w} \sum_i w_i \skal{\tilde{A}e_i, \tilde{v}}\big)^2 
	\le 4 \sup_{A \in \mathcal{A}} \norm{A}_{\mathrm{op}}^2.
	\end{align*}
	Here, the suprema of $v$ and $w$ are taken over the $n$-dimensional sphere. 
	We can now apply Corollary \ref{corollary:concIneqAlteVersion} to $\Gamma = \dpartial^+$, $\rho = 2\sigma^2$, $g = 4f_{\mathcal{A}}$ and $b = 8 \Sigma$ to finish the proof.
\end{proof}

\begin{proof}[Proof of Proposition \ref{proposition:dPlusModLSI}]
	The idea of the proof of the $\mathrm{mLSI}$s is already present in \cite{BG07}. Let $(\Omega, \mathcal{F}, \nu)$ be any probability space. For any function $g$ we have due to the inequality $(a-b)_+(e^a - e^b)_+ \le \frac{1}{2} (a-b)_+^2 (e^a + e^b)$ (for all $a,b \in \IR$)
	\begin{align*}
	\mathrm{Cov}_\nu(g, e^g) 
	&\le \frac{1}{2} \iint (g(x) - g(y))_+^2(e^{g(x)} + e^{g(y)}) d\nu(x) d\nu(y)\\
	&= \frac{1}{2} \iint (g(x) - g(y))^2 d\nu(y) e^{g(x)} d\nu(x).
	\end{align*}
	Applying this to $\nu = \mu(\cdot \mid x_{i^c})$ and $g = f(x_{i^c}, \cdot)$ and using \eqref{eqn:modLSI} yields
	\begin{align*}
	\mathrm{Ent}_{\mu}(e^f) \le \frac{\sigma^2}{2} \sum_{i = 1}^n \iint (f(x) - f(x_{i^c}, x_i'))^2 d\mu(x_i' \mid x_{i^c}) e^{f(x)} d\mu(x) = \frac{\sigma^2}{2} \int \abs{\dpartial f}^2 e^f d\mu.
	\end{align*}
	To see that $\mu$ also satisfies a $\mathfrak{d^+}\mathrm{-mSLI} (2\sigma^2)$, it remains to apply Lemma \ref{mLSIs}. The exponential inequalities are a consequence of Theorem \ref{theorem:BG99}.
\end{proof}

\begin{proof}[Proof of Theorem \ref{theorem:Bernstein}]
	Write $X = (X_1, \ldots, X_n)$. Let us assume that $\IE X_i = 0$ for all $i \in \{1,\ldots, n\}$, from which the general case follows easily using the inequality
	\[
	\norm{\max_i \abs{X_i - \IE X_i}}_{\Psi_2} \le 4 \norm{\max_i \abs{X_i}}_{\Psi_2}.
	\]
	
	Since the $X_i$ are independent, it follows from Proposition \ref{proposition:dPlusModLSI} that their joint distribution $\mathbb{P}_X$ satisfies a $\dpartial\mathrm{-mLSI}(1)$, and we can calculate
	\begin{align*}
	\abs{\dpartial f}(X) &= \Big( \sum_{i = 1}^n \int (X_i - y_i)^2 d\IP_{X_i}(y_i) \Big)^{1/2} = \Big( \sum_{i = 1}^n X_i^2 + \IE X_i^2 \Big)^{1/2} \\
	&\le \abs{\norm{X}_2 - \IE \norm{X}_2} + \IE \norm{X}_2 + \Big( \sum_{i =1}^n \IE X_i^2 \Big)^{1/2} \eqqcolon g(X).
	\end{align*}
	To apply Theorem \ref{theorem:SecondOrderConcentration}, it remains to show that we may set $c = \IE \norm{X}_2 + \sqrt{\mathrm{Var}(\sum_i X_i)}$ and $K = \norm{\max_i \abs{X_i}}_{\Psi_2}$. This is seen by noting that
	\[
	\IP(g(X) \ge c + t) = \IP(\abs{\norm{X}_2 - \IE \norm{X}_2} \ge t) \le 2\exp\Big( -c_2 \frac{t^2}{\norm{\max_i \abs{X_i}}_{\Psi_2}^2} \Big),
	\]
	where the last step follows from \cite[Lemma 1.4]{KZ18}, as $X \mapsto \norm{X}_2$ is a convex and $1$-Lipschitz function. Note that although \cite[Lemma 1.4]{KZ18} is formulated for $t \ge t_0 > 0$, one can easily find an estimate for all $t \ge 0$, by first multiplying the right hand side by $2$, and then adjusting the constant in the exponential.
\end{proof}

Recall that as discussed above, the application of Theorem \ref{theorem:BG99} is only possible for bounded functions, so that an additional truncation step needs to be done. Instead of applying Theorem \ref{theorem:BG99} to $f(X) = \sum_i X_i - \IE X_i$, it is applied to the sum of the random variables $Y_i \coloneqq g_R(X_i) - \IE g_R(X_i)$ for $g_R(x) = \min(R,\max(x,-R))$ for a suitable truncation level $R > 0$. As the right hand side of equation \eqref{eqn:BernsteinIneq} can be chosen to be independent of $R$, the theorem follows for unbounded random variables by letting $R \to \infty$.

\printbibliography

 \end{document}